\documentclass[twoside,11pt]{article}

\usepackage{blindtext}
\usepackage[preprint, abbrvbib]{jmlr2e}
\usepackage[utf8]{inputenc} 
\usepackage[OT1]{fontenc}    
\usepackage{hyperref}       
\usepackage{url}            
\usepackage{booktabs}       
\usepackage{amsfonts}       
\usepackage{nicefrac}       
\usepackage{microtype}      

\usepackage{mathtools}
\usepackage{mathrsfs}
\usepackage{graphicx}
\usepackage{dsfont}
\usepackage{wrapfig}
\usepackage{comment}
\usepackage{rotating}
\usepackage{subcaption}
\usepackage[font={small}]{caption}  
\usepackage{algorithmic}
\usepackage[ruled,linesnumbered,shortend,vlined,noend]{algorithm2e}
\usepackage[all]{xy}
\usepackage{xcolor}
\usepackage[capitalize]{cleveref}
\usepackage{arydshln}
\usepackage{enumitem}
\setlist[itemize]{leftmargin=*}
\usepackage{wrapfig}  
\usepackage{listings}

\DeclareMathOperator*{\argmin}{arg\,min}

\lstdefinestyle{cust}{
language=python,
commentstyle=\ttfamily,
basicstyle=,
escapeinside={\%*}{*)},
frame=single,
keepspaces=true,
keywordstyle=\bfseries,
morekeywords={*,Input,Output},
}
\renewcommand{\BlackBox}{\rule{1.5ex}{1.5ex}}  
\ifdefined\proof
    \renewenvironment{proof}{\par\noindent{\bf Proof\ }}{\hfill\BlackBox\\[2mm]}
\else
    \newenvironment{proof}{\par\noindent{\bf Proof\ }}{\hfill\BlackBox\\[2mm]}
\fi

\newtheorem{lem}[theorem]{Lemma}
\newtheorem{defi}[theorem]{Definition}
\newtheorem{prop}[theorem]{Proposition}

\newcommand{\cb}{\mathbf{c}}
\newcommand{\R}{\mathbb{R}}
\newcommand{\E}{\mathbb{E}}

\renewcommand{\1}{\mathds 1}

\newcommand{\N}{\mathbb{N}}
\newcommand{\B}{\mathcal{B}}

\renewcommand{\P}{\mathbb{P}}
\newcommand{\Z}{\mathbb{Z}}

\newcommand{\Cor}{\operatorname{Cor}}
\newcommand{\Rips}{\mathrm{VR}}
\newcommand{\indep}{{\perp \!\!\! \perp}}
\renewcommand{\d}{\mathrm{d}}
\newcommand{\supp}{\mathrm{Supp}}
\newcommand{\clement}[1]{{\color{red}{Clément: #1}}}

\newcommand{\floor}[1]{\left\lfloor #1 \right\rfloor}

\SetKw{Init}{Initialization:}

\begin{document}

\title{Topological Analysis for Detecting Anomalies (TADA)\\ in Time Series.}
\author{Fr\'ed\'eric Chazal, Cl\'ement Levrard, and Martin Royer}
\date{}
\maketitle
\ShortHeadings{TADA}{Topological Analysis for Detecting Anomalies in Time Series}
\firstpageno{1}
%

\begin{abstract}

This paper introduces new methodology based on the field of Topological Data Analysis for detecting anomalies in multivariate time series, that aims to detect global changes in the dependency structure between channels. The proposed approach is lean enough to handle large scale datasets, and extensive numerical experiments back the intuition that it is more suitable for detecting global changes of correlation structures than existing methods. Some theoretical guarantees for quantization algorithms based on dependent time sequences are also provided.

\end{abstract}


\section{Introduction}

Monitoring the evolution of the global structure of time-dependent complex data, such as, e.g., multivariate times series or dynamic graphs, is a major task in real-world applications of machine learning. 
The present work considers the case where the global structure of interest is a weighted dynamic graph encoding the dependency structure between the different channels of a multivariate time series. Such a situation may be encountered in various fields, such as e.g. EEG signal analysis \cite{Mohammed23} or monitoring of industrial processes \cite{li2022anomaly}, and has recently given rise to an abundant literature - see, e.g. \cite{zheng2023correlation,ho2023graph} and references therein.


The specific monitoring task addressed in this paper is unsupervised anomaly detection, that is to detect when the dependency structure is far enough from a so-called "normal" regime to be considered as problematic. From the mathematical point of view, this problem, in its whole generality, is ill-posed: one has access to unlabeled data, in which it is tacitly assumed that the normal regime is prominent, the goal is then to label data points as normal or abnormal in a fully unsupervised way. In this sense, anomaly detection shows clear connection with outlier detection in robust machine learning (for instance robust clustering as in \cite{Brecheteau21,Jana24}). For more insights and benchmarks on the specific problem of anomaly detection in times series the reader is referred to \cite{paparrizos22} (univariate case) and to \cite{WenigEtAl2022TimeEval} for the multivariate case. 

We introduce a new framework, coming with mathematical guarantees, based upon the use of Topological Data Analysis (TDA), a field that has know an increasing interest to  study complex data - see, e.g. \cite{chazal2021introduction} for a general introduction. Application of TDA to anomaly detection in time series have raised a recent and growing interest: in medicine (\cite{Dindin19,Petri14,Chretien24}), cyber security (\cite{Bruillard16})... And, some general surveys on TDA applications to time series may be found in \cite{Ravishanker19,Umeda19}.

In this paper, the adopted approach proceeds in three steps. First, the time-dependency structure of a time series is encoded as a dynamic graph in which each vertex represents a channel of the time series and each weighted edge encodes the dependency between the two corresponding vertices over a time window. Second, persistent homology, a central theory in TDA, is used to robustly extract the global topological structure of the dynamic graph as a sequence of so-called persistence diagrams. 
Third, we introduce a specific encoding of persistence diagrams, that has been proven efficient and simple enough to face large-scale problems in the independent case (\cite{Chazal21}), to produce a topological anomaly score.
As detailed throughout the paper, the scope of the proposed method may be extended in several ways, encompassing dependent sequences of measures and dependent sequences of general metric spaces.

\subsection{Contributions}

Our main contributions are the following.
\begin{itemize}
\item[-] We produce a new machine learning methodology for learning the normal topological behavior in the data spatial dependency structure. This methodology is fully unsupervised, it does not need to be calibrated on uncorrupted data, as long as the amount of corrupted data remains limited with respect to the uncorrupted one. The captured information is numerically proved different and, in several cases, more informative than the one captured by other state-of-the-art approaches. This methodology is lean by design, and enjoys novel interpretable properties with regards to anomaly detection that have never appeared in the literature, up to our knowledge;
\item[-] The proposed pipeline is easy to implement, flexible and can be adapted to different specific applications and framework involving graph data or more general topological data; 
\item[-] The resulting method can be deployed on architectures with limited computational and memory resources: once the training phase realized, the anomaly detection procedure relies on a few memorized parameters and simple persistent homology computations. Moreover, this procedure does not require any storage of previously processed data, preventing privacy issues; 
\item[-] Some convergence guarantees for quantization algorithms - used to vectorize topological information - in a dependent case are proven. These results do not restrict to the specific setting of the paper and may be generalized in the general framework of $M$-estimation with dependent observations;
\item[-] Extensive numerical investigation has been carried out in three different frameworks. First on new synthetic data directly inspired from brain modelling problems as exposed in \cite{Bourakna22}, that are particularly suited for TDA-based methods and may be used as novel benchmark. They are added to the public \cite{gudhi} library at \url{github.com/GUDHI/gudhi-data}. Second on the comprehensive benchmark "TimeEval", that encompasses a large array of synthetic datasets \cite{SchmidlEtAl2022Anomaly}. And third on a real-case "Exathlon" dataset from \cite{exathlon}.
All of these experiments assess the relevance of our approach compared with current state-of-the-art methods. Our procedure, originating from concrete industrial problems, is implemented and has been deployed within the Confiance.ai program and an open-source release is incoming. Its implementation involves only standard, tested machine learning tools.
\end{itemize}

\subsection{Organization of the paper}

A complete description of the proposed methodology in provided in Section \ref{sec:methodo}, giving details on the several steps to build an anomaly score from a multivariate time series. Next, Section \ref{sec:theoretical_results} theoretically grounds the centroid computation step as well as the anomaly test proposed in the previous section. Section \ref{sec:applications} gathers the numerical experiments in the three different settings introduced above (synthetic TDA-friendly, TimeEval synthetic, real Exathlon data). Proofs of our results are postponed to Section \ref{sec:proofs}.

\section{Methodology}\label{sec:methodo}

This section describes the pipeline to build an anomaly score from raw multivariate time series data $(Y_t)_{t \in [0,L]} \in \R^D \times [0,L]$. We start with a brief description of the TDA tools that we use.

\subsection{Vietoris-Rips persistent homology for weighted graphs}\label{sec:methodo_persistence}
In this section we briefly explain how discrete measures are associated to weighted graphs, encoding their multiscale topological structure through persistent homology theory. We refer the reader to \cite{Edelsbrunner2010,Chazal2016,boissonnat2018geometric} for a general and thorough introduction to persistent homology.   

Recall that given a set V , an (abstract) simplicial complex is a set K of finite subsets of V such that $\sigma \in K$ and $\tau \subset \sigma$ implies $\tau \in K$. Each set $\sigma \in K$ is called a simplex of $K$. The dimension of a simplex $\sigma$ is defined as $|\sigma| - 1$ and the dimension of $K$ is the maximum dimension of any of its simplices. Note
that a simplicial complex of dimension $1$ is a graph. A simplicial complex classically inherits a canonical structure of topological space obtained by representing each simplex by a geometric standard simplex (convex hull of a finite set of affinely independent points in an Euclidean space) and ``gluing'' the simplices along common faces.
A filtered simplicial complex $(K_\alpha)_{\alpha \in I}$, or filtration for short, is a nested family of complexes indexed by a set of real numbers $I \subset \mathbb{R}$: for any $\alpha, \beta \in I$, if $\alpha \leqslant \beta$ then $K_\alpha \subseteq K_\beta$.  The parameter $\alpha$ is often seen as a scale parameter.

Let $G$ be a complete non-oriented weighted graph with vertex set $V$ and real valued edge weight function $s : V \times V \to \mathbb{R}$, $(v,v') \mapsto s_{v,v'}$, satisfying $s_{v,v'} := s_{v',v}$ for any pair of vertices $(v,v')$.

\begin{defi}
Let $\alpha_{\min} \leqslant \min_{v,v' \in V} s_{v,v'}$ and $\alpha_{\max} \geqslant \max_{v,v' \in V} s_{v,v'}$ be two real numbers. 
The Vietoris-Rips filtration associated to $G$ is the filtration $(\Rips_\alpha(G))_{\alpha \in [\alpha_{\min},\alpha_{\max}]}$ with vertex set $V$ defined by
$$\sigma = [v_{0}, \cdots, v_{k}] \in \Rips_\alpha(G) \ \ \mbox{\rm if and only if} \ \ s_{v_i,v_j} \leqslant \alpha, \ \ \mbox{\rm for all} \ \ i,j \in [\![0,k]\!],$$
for $k>1$, and $[v] \in \Rips_\alpha(G)$ for any $v \in V$ and any $\alpha \in  [\alpha_{\min},\alpha_{\max}]$.
\end{defi}

The topology of $\Rips_\alpha(G)$ changes as $\alpha$ increases: existing connected components may merge, loops and cavities may appear and be filled, etc...  Persistent homology provides a mathematical framework and efficient algorithms to encode this evolution of the topology (homology) by recording the scale parameters at which topological features appear and disappear. Each such feature is then represented as an interval $[\alpha_b,\alpha_d]$ representing its life span along the filtration. Its length $\alpha_d - \alpha_b$ is called the persistence of the feature.  
The set of all such intervals corresponding to topological features of a given dimension $d$ - $d=0$ for connected components, $d=1$ for $1$-dimensional loops, $d=2$ for $2$-dimensional cavities, etc... -  is called the persistence barcode of order $d$ of $G$. It is also classically represented as a discrete multiset $D_d(G) \subset [\alpha_{\min},\alpha_{\max}]^2$ where each interval $[\alpha_b,\alpha_d]$ is represented by the point with coordinates $(\alpha_b,\alpha_d)$ - a basic example is given on Figure \ref{fig:RipsDiags}. Adopting the perspective of \cite{ChazalDivol18,Royer19, Chazal21}, in the sequel of the paper, the persistence diagram $D_d(G)$ will be considered as a discrete measure: $D_d(G) := \sum_{p \in D_d(G)} \delta_p$ where $\delta_p$ is the Dirac measure centered 
at $p$. In many practical settings, to control the influence of the, possibly many, low persistence features, the atoms in the previous sum can be weighted: 
$$D_d(G) := \sum_{(b,d) \in D_d(G)} \omega(b,d) \delta_{(b,d)}, $$
where $\omega: \mathbb{R}^2 \to \mathbb{R}_+$ may either be a continuous function which is equal to $0$ along the diagonal or just a constant renormalization factor equal to the total mass of the diagram.
Notice that, in practice, there exist various libraries to efficiently compute persistence diagrams, such as, e.g., \cite{gudhi}. 

\begin{figure}[h!]
\centering
\includegraphics[width=.8\linewidth]{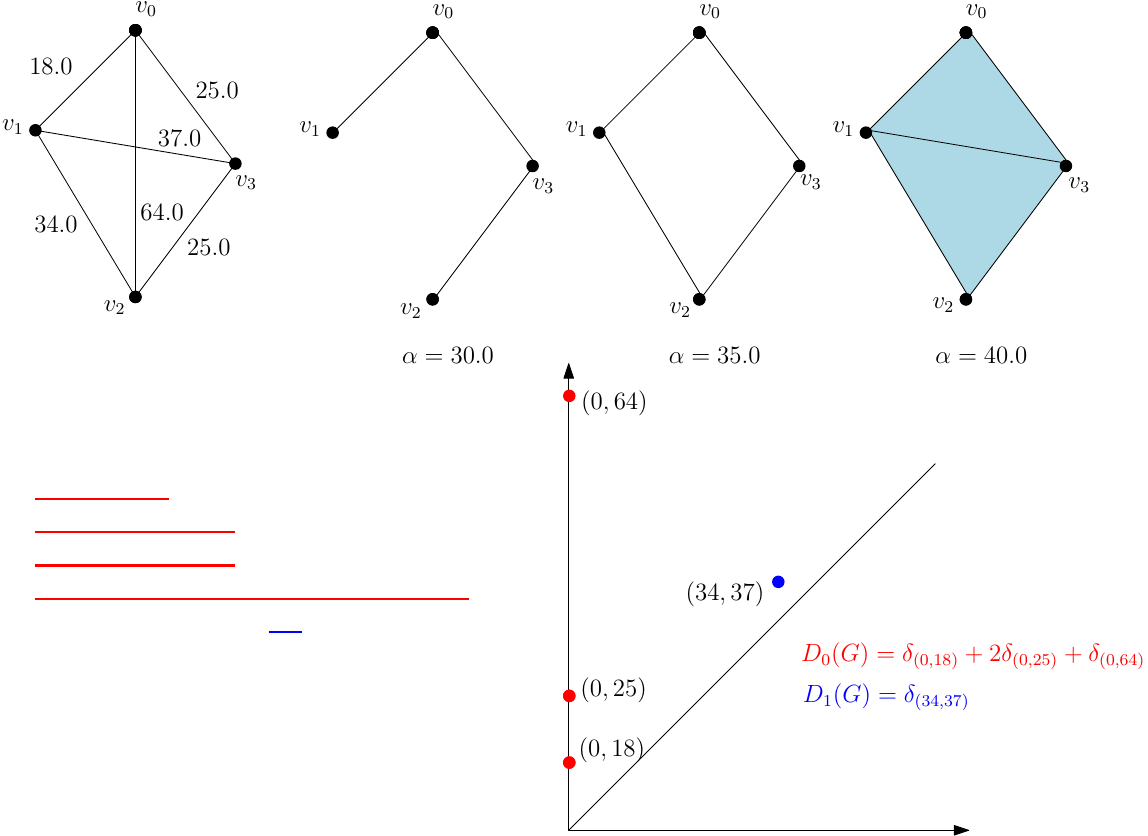}
\caption{The persistence diagrams of order $0$ and $1$ of a simple weighted graph $G$ whose vertices are $4$ points in the real plane and edge weigths are given by the squared distances between them. Here $\alpha_{\min}$ and $\alpha_{\max}$ are chosen to be $0$ and $64$ respectively. The first line represents $G$ and $\Rips_\alpha(G)$ for different values of $\alpha$. The persistence barcodes and diagrams of order $0$ and $1$ are represented in red and blue respectively on the second line.}
\label{fig:RipsDiags}
\end{figure}

The relevance of the above construction relies on the persistence stability theorem \cite{Chazal2016}. It ensures that close weighted graphs have close persistence diagrams. More precisely, if $G, G'$ are two weighted graphs with same vertex set $V$and edge weight functions $s: V \times V \to \mathbb{R}$ and $s': V \times V \to \mathbb{R}$ respectively, then for any order $d$, the so-called bottleneck distance between the persistence diagrams $D_d(G)$ and $D_d(G')$ is upperbounded by $\| s - s' \|_\infty := \sup_{v,v' \in V} |s_{v,v'} - s'_{v,v'}|$ - see \cite{chazal2014persistence} for formal persistence stability statements for Vietoris-Rips complexes.

\subsection{From similarity matrices to persistence diagrams}\label{sec:persistence_diagrams_computation}

Our first step is to extract the topological information pertaining to the dependency structure between channels. To do so, for a window size $\Delta$, the $D$-dimensional time serie is sliced into $n$ sub-intervals. For each sub-interval $[st, st+\Delta]$, we build a coherence graph $G_t$, starting from the fully-connected graph $([\![1,D]\!],E)$ and specifying edge values as $s_{i,j,t} = 1 - \Cor_t(Y_i,Y_j)$, that is $1$ minus the correlation between channels $i$ and $j$ computed in the interval $t$.

Then, the persistence diagrams of the Vietoris-Rips filtration are computed (one per homology order), resulting in sequences of diagrams $X_t^{(d)}$, with $ 1 \leqslant t \leqslant n$ and $d$ is the homological order. An example of sequences of windows and corresponding persistence diagrams is represented Figure \ref{fig:sw-pdiags}.
In what follows, a fixed homology order is considered, so that the index $d$ is removed. In practice, the vectorization steps that follows are performed order-wise, as well as the anomaly detection procedure.

\begin{algorithm}[H]
	\SetAlgoLined
	\KwIn{$p$ maximal homology order, $\Delta$ window size, $s$ stride.}
	\KwData{A multivariate time serie $(Y_t)_{t \in [0,L]} \in \mathbb{R}^D\times [0,L]$}
	\For{$t$ in $[\![0, \floor{(L-\Delta)/s} ]\!]$}{
		compute similarity matrix on the slice $[st, st+\Delta]$, $S_t = 1-Corr(Y_{[st, st+\Delta]})$\;
		compute the Vietoris-Rips filtration for $([\![1,D]\!],E, S_t)$ \;
		\For{homology dimension $d$ in $[\![0, p-1]\!]$}{
		compute order $d$ persistence diagram $X_t^{(d)}$ of the Rips filtration\;
		}
		}
		\KwOut{$p$  (discrete) time series of persistence diagrams $X_t^{(d)}$, $t \in [\![ 0, \floor{(L-\Delta)/s} ]\!]$, $d\in[\![0, p-1]\!]$.}
	\caption{Persistence diagrams computation from a multidimensional time serie}
	\label{algo:timeserie_to_pers_diagrams}
\end{algorithm}

It is worth noting that other dependency measures such as the ones based on coherence \cite{OmbaO21} may be chosen instead of correlation to build the weighted graphs. Such alternative choices do not affect the overall methodology as well as the theoretical results provided below. 

In the numerical experiments, we give results for the correlation weights, that have the advantages of simplicity and carry a few insights: following \cite{Bourakna22}, such weights are enough to detect structural differences in the case where the channels $Y_j$ are mixture of independent components $Z_p$'s, the weights of the mixture being given by a (hidden) graph on the $p$'s whose structure drives the behavior of the observed persistence diagrams.

At last, it is worth recalling here that Vietoris-Rips filtration may be built on top of arbitrary metric spaces, so that the persistence diagram construction may be performed in more general cases, encompassing valued graphs (with value on nodes or edges) for instance. The vectorization and detection steps below being based on the inputs of such persistence diagrams, the scope of our approach is easily extended beyond the analysis of multivariate time series.

\begin{figure}[h!]
	\centering
	\begin{subfigure}[m]{.45\linewidth}
		\centering
		\includegraphics[width=\linewidth]{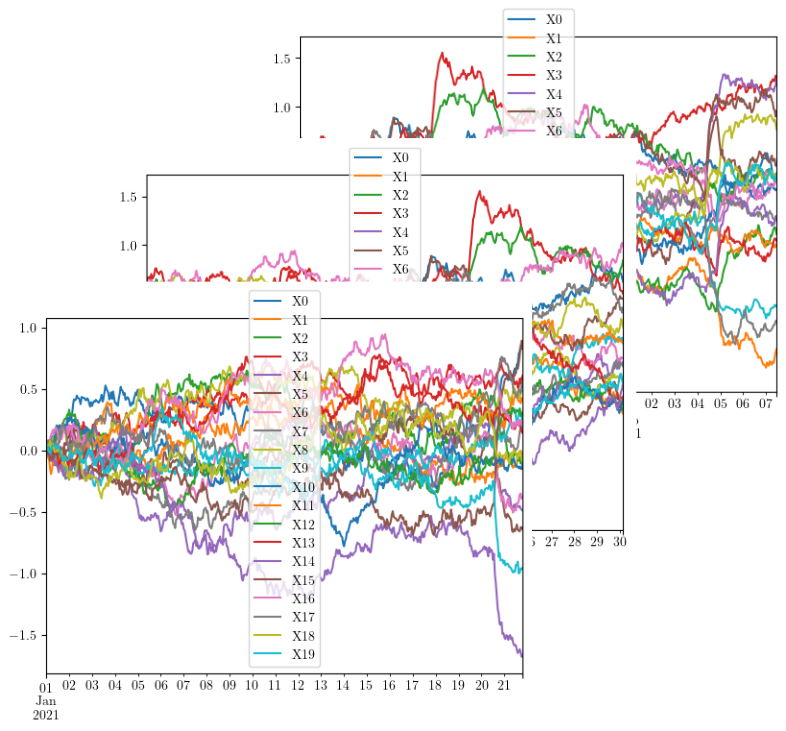}
	\end{subfigure}
	\begin{subfigure}[m]{.45\linewidth}
		\includegraphics[width=\linewidth]{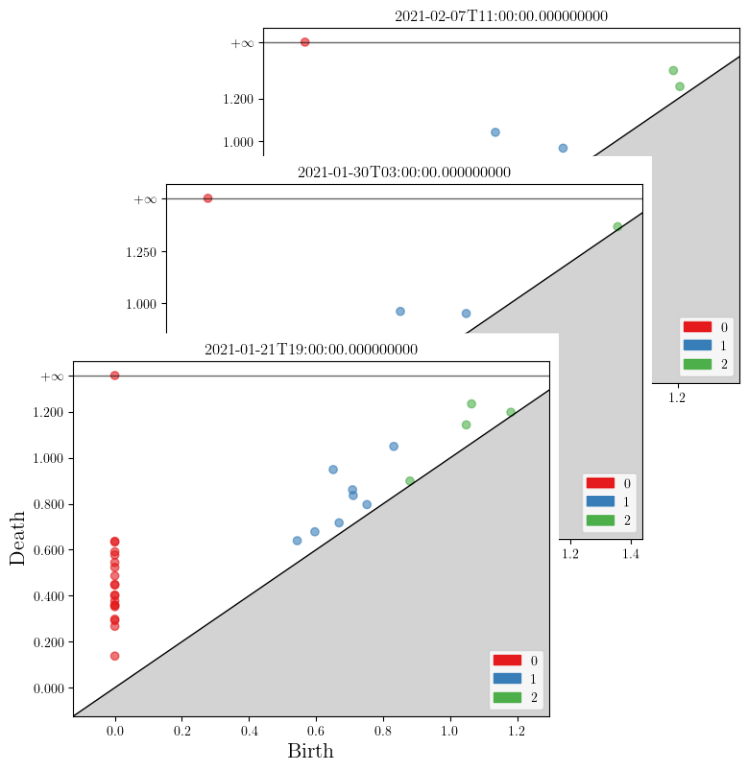}
	\end{subfigure}
	\caption{Left: three sliding windows on an illustrative Ornstein–Uhlenbeck (AR1) synthetic process with additive, punctual anomalies. Right: corresponding to those sliding windows, the three topological descriptors (persistence diagrams with Homology dimension 0 (red), 1 (blue) and 2 (green) features), according to Algorithm \ref{algo:timeserie_to_pers_diagrams}.}
	\label{fig:sw-pdiags}
\end{figure}

\subsection{Centroids computation}\label{sec:centroids}

Once the mutivariate time series are converted into a sequence of persistence diagrams $(X_i)_{i=1, \hdots, n}$, the next step is to convert these persistence diagrams into vectors, that roughly encode how much mass do these diagrams spread into well-chosen areas of $\R^2$. To do so, we begin by automatically choosing these areas, or equivalently centers of these sub-areas, using the ATOL procedure of measure quantisation from \cite{Royer19}.

\begin{algorithm}[H]
	\SetAlgoLined
	\KwIn{$K$: number of centroids. $T$: stopping time.}
	\KwData{$X_1, \hdots, X_n$ discrete measures.}
	\Init{$\cb^{(0)} = (c_1^{(0)}, \hdots, c_K^{(0)})$} randomly chosen from $\bar{X}_n$.
	
	\For{$t=1, \hdots, T$}{
		\For{$j=1, \hdots, K$}{
			$W_{j,t-1} \leftarrow \{x \in \R^2 \mid \forall i \neq j \quad \| x-c_j^{(t-1)}\| \leqslant \|x-c_i^{(t-1)}\|\}$ (ties arbitrarily broken).
			
			\If{$ \bar{X}_n(W_{j,t-1}) \neq 0$}{
				$c_j^{(t)} \leftarrow (\bar{X}_n(du) \left ( u \1_{W_{j,t-1}}(u) \right ))/\bar{X}_n(W_{j,t-1})$.
			}
			\Else{$c_j^{(t)} \leftarrow \mbox{random sample from $\bar{X}_n$}$.}
		}
	}
	\KwOut{Centroids $\cb^{(T)} = (c_1^{(T)}, \hdots, c_K^{(T)})$.}
	\caption{Centroid computations - ATOL - Batch algorithm}
	\label{algo:kmeans_batch}
\end{algorithm}

The batch algorithm for computing centers is recalled below. As introduced in Section \ref{sec:methodo_persistence}, a persistence diagram $X_i$ is thought of as a discrete measure on $\R^2$, that is 
\begin{align*}
X_i = \sum_{(b,d) \in PD_i} \omega_{(b,d)}\delta_{(b,d)},
\end{align*}
where $PD_i$ is the $i$-th persistence diagram considered as a multiset of points, and $\omega_{(b,d)}$ are weights given to points in the persistence diagram (usually given as a function of the distance from the diagonal, see e.g., \cite{Adams17} for instance).  For the batch algorithm, a special interest is paid to the empirical mean measure:
\begin{align*}
\bar{X}_n = \frac{1}{n} \sum_{i=1}^n X_i.
\end{align*}

Algorithm \ref{algo:kmeans_batch} is the same as in the i.i.d. case \cite[Algorithm 1]{Chazal21}. Moreover, almost the same convergence guarantees as in the i.i.d. case may be proven: for a good-enough initialization, only $2 \log(n)$ iterations are needed to achieve a statistically optimal convergence (see Theorem \ref{thm:cv_batch} below). Therefore, a practical implementation of Algorithm \ref{algo:kmeans_batch} should perform several threads based on different  initializations (possibly in parallel), each of them being stopped after $2 \log(n)$ steps, yielding a complexity in time of $O( n \log(n) \times n_{start})$ (where $n_{start}$) is the number of threads.

As for the i.i.d. case, an online version of Algorithm \ref{algo:kmeans_batch} may be conceived, based on mini-batches. In what follows, for a convex set $C \subset \R^d$, we let $\pi_{C}$ denote the Euclidean projection onto~$C$.

\begin{algorithm}[H]
	\SetAlgoLined
	\KwIn{$K$: number of centroids. $q$: size of mini-batches. $R$: maximal radius.}
	\KwData{$X_1, \hdots, X_n$ discrete measures.}
	\Init{$\cb^{(0)} = (c_1^{(0)}, \hdots, c_K^{(0)})$} randomly chosen from $\bar{X}_n$.
	Split $X_1, \hdots, X_n$ into $ n/q$ mini-batches of size $q$: $B_{1,1}, B_{1,2}, B_{1,3}, B_{1,4}, \hdots, B_{t,1}, B_{t,2}, B_{t,3}, B_{t,4}, \hdots, B_{T,1}, B_{T,2}, B_{T,3}, B_{T,4}$, $T=n/4q$.
	
		\For{$t=1, \hdots, T$}{
			\For{$j=1, \hdots, K$}{
			$W_{j,t-1} \leftarrow \{x \in \R^2 \mid \forall i \neq j \quad \| x-c_j^{(t-1)}\| \leqslant \|x-c_i^{(t-1)}\|\}$ (ties arbitrarily broken).
			
			\If{$ \bar{X}_{B_{t,1}}(W_{j,t-1}) \neq 0$}{
			$c_j^{(t)} \leftarrow \pi_{\B(0,R)} \left ( (\bar{X}_{B_{t,3}}(du) \left ( u \1_{W_{j,t-1}}(u) \right ))/\bar{X}_{B_{t,1}}(W_{j,t-1}) \right )$.
			}
			\Else{$c_j^{(t)} \leftarrow c_j^{(t-1)}$.}
			}
		}
	\KwOut{Centroids $\cb^{(T)} = (c_1^{(T)}, \hdots, c_K^{(T)})$.}
	\caption{Centroid computations - ATOL - Minibatch algorithm}
	\label{algo:kmeans_minibatch}
\end{algorithm}

Contrary to Algorithm \ref{algo:kmeans_batch}, Algorithm \ref{algo:kmeans_minibatch} differs from its i.i.d. counterpart given in \cite{Chazal21}. First, the theoretically optimal size of batches is now driven by the decay of the $\beta$-mixing coefficients of the time serie, as will be made clear by Theorem \ref{thm:cv_minibatch} below. 

Second, half of the sample are wasted (the $B_{t,j}$'s with even $j$). This is due to theoretical constraints to ensure that the mini-batches that are used are spaced enough to guarantee a prescribed amount of independence. Of course, the even $B_{t,j}$'s could be used to compute a parallel set of centroids. However, in the numerical experiments, all the sample is used (without leaving some space between minibatches), without noticeable side effect. 

From a computational viewpoint, Algorithm \ref{algo:kmeans_minibatch} is single-pass, so that, if $n_{start}$ threads are run, the global complexity is in $O(n \times n_{start})$.

See an instance of the centroid computations on Figure \ref{fig:sw-atolcenters}.

\begin{figure}[h!]
	\centering
	\begin{subfigure}[m]{.45\linewidth}
		\centering
		\includegraphics[width=\linewidth]{oua-perss.png}
	\end{subfigure}
	\begin{subfigure}[m]{.45\linewidth}
		\centering
		\includegraphics[width=\linewidth]{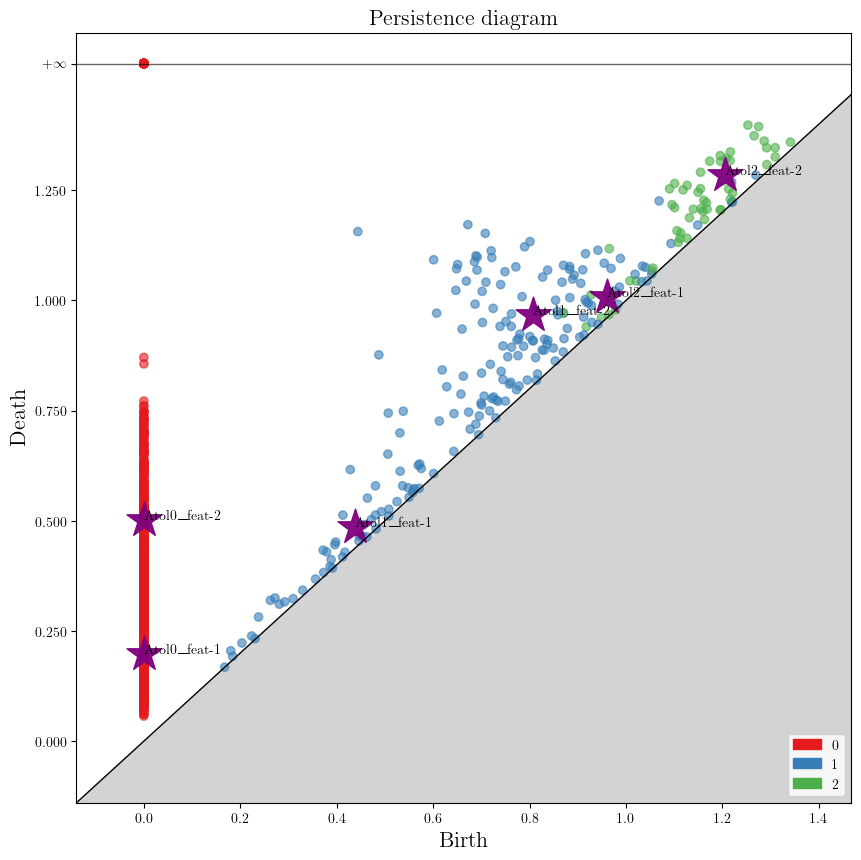}
	\end{subfigure}
	\caption{Left: three representative topological descriptors in the form of persistence diagrams with Homology dimensions 0, 1 and 2. Right: sum of topological descriptors and their centroids (stars in purple, two by dimension) computed from them in dimensions 0, 1 and 2 according to Algorithm \ref{algo:kmeans_batch} or Algorithm \ref{algo:kmeans_minibatch}.}
	\label{fig:sw-atolcenters}
\end{figure}

\subsection{Conversion into vector-valued time series}

Once the centroids $\cb^{(T)}$ built, the next step is to convert the persistence diagrams $(X_i)_{i=1, \hdots, n}$ into vectors. The approach here is the same as in \cite{Royer19} denoting by $\psi_{AT}: u \mapsto \exp(-u^2)$, a persistence diagram $X_i$ is mapped onto
\begin{align}\label{eq:ATOL_vectorization}
v_i = \left ( X_i(du) \psi_{AT}(\|u-c_1^{(T)}\|/\sigma_1), \hdots, X_i(du) \psi_{AT}(\|u-c_K^{(T)}\|/\sigma_K) \right ),
\end{align}
where the bandwiths $\sigma_j$'s are defined by
\begin{align}\label{eq:ATOL_bandwith}
\sigma_j = \min_{\ell \neq j} \|c^{(T)}_\ell - c^{(T)}_j\|/2, 
\end{align}
that roughly seizes the width of the area corresponding to the centroid $c^{(T)}_j$. Other choices of kernel $\psi$ are possible (see e.g. \cite{Chazal21}), as well as other methods for choosing the bandwidth. The proposed approach has the benefit of not requiring a careful parameter tuning step, and seems to perform well in practice.

We encapsulate this vectorization method as follows, and an example vectorization is shown in Figure \ref{fig:sw-atolvect}.

\begin{algorithm}[H]\label{algo:vec}
	\SetAlgoLined
	\KwIn{Centroids $c_1, \hdots, c_K$}
	\KwData{A persistence diagram $X$}
	\For{$j=1, \hdots, K$}{
	Compute $\sigma_j$ as in \eqref{eq:ATOL_bandwith};
	
	$v_j \leftarrow X(du) \psi_{AT}(\|u-c_j\|/\sigma_j)$}
	\KwOut{Vectorization $v=(v_1, \hdots, v_K)$.}
	\caption{Vectorization step}
\end{algorithm}

\begin{figure}[h!]
	\centering	
	\begin{subfigure}[m]{.45\linewidth}
		\centering
		\includegraphics[width=\linewidth]{oua-centers.png}
	\end{subfigure}
	\begin{subfigure}[m]{.45\linewidth}
		\centering	\includegraphics[width=\linewidth]{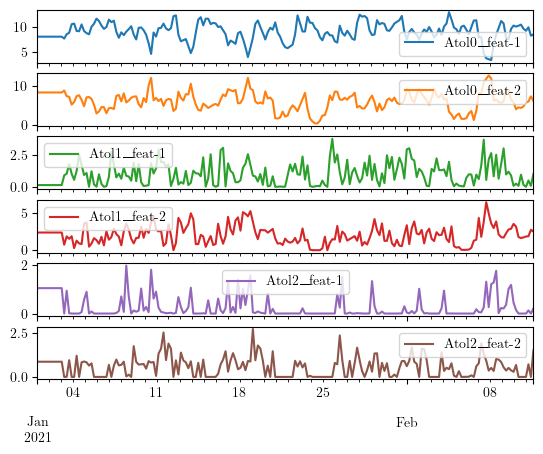}
	\end{subfigure}
	\caption{Left: sum of topological descriptors and their centroids (stars in purple, two by dimension) computed from them in dimensions 0, 1 and 2 by Algorithm \ref{algo:kmeans_batch} or \ref{algo:kmeans_minibatch}. Right: the derived topological vectorization of the entire time series computed relative to each center according to Algorithm \ref{algo:vec}.}
	\label{fig:sw-atolvect}
\end{figure}

\subsection{Anomaly detection procedure}\label{sec:anomaly_detection}

We assume now that we observe the vector-valued time serie $v_1, \hdots, v_n$ of vectorized
persistence diagrams, and intend to build a procedure to determine whether a new diagram (processed with Algorithm \ref{algo:vec}) may be thought of as an anomaly.

Our first step is to build a score, based on the "normal" behaviour of the vectorizations $v_i$'s that are thought of as originating from a base regime. Namely, we build the sample means and covariances
\begin{align}\label{eq:mean_cov_vectorizations}
\hat{\mu} & = \frac{1}{n} \sum_{i=1}^n v_i, \notag \\
\hat{\Sigma} & = \frac{1}{n} \sum_{i=1}^n (v_i - \hat{\mu})(v_i - \hat{\mu})^T.
\end{align}
In the case where the base sample can be corrupted, robust strategies for mean and covariance estimation such as \cite{Rousseeuw99, Mia18} may be employed. To be more precise, for a contamination parameter $h \in [0; 1]$,  we choose the Minimum Covariance Determinant estimator (MCD), defined by
\begin{align}\label{eq:mean_cov_vectorizations_minCovDet}
\hat{I} & \in \argmin_{ I \subset \{1, \hdots, n\}, |I| = \lceil n(1-h) \rceil } \mathrm{Det} \left (  \frac{1}{|I|} \sum_{i \in I} (v_i - \bar{v}_I)(v_i - \bar{v}_I)^T \right ), \notag \\ 
\hat{\mu} & = \bar{v}_{\hat{I}}, \notag \\ 
\hat{\Sigma} & = c_0 \left (  \frac{1}{|\hat{I}|} \sum_{i \in \hat{I}} (v_i - \hat{\mu})(v_i - \hat{\mu})^T \right ),
\end{align}
where $\bar{v}_I$ denotes empirical mean on the subset $I$, and $c_0$ is a normalization constant that can be found in \cite{Mia18}. In all the experiments exposed in Section \ref{sec:applications}, we use the approximation of \eqref{eq:mean_cov_vectorizations_minCovDet} provided in \cite{Rousseeuw99}.

\begin{figure}[h!]
	\centering
	\begin{subfigure}[m]{.45\linewidth}
		\centering	\includegraphics[width=\linewidth]{oua-topol-embedding.png}
	\end{subfigure}
	\begin{subfigure}[m]{.45\linewidth}
		\centering
		\includegraphics[width=\linewidth]{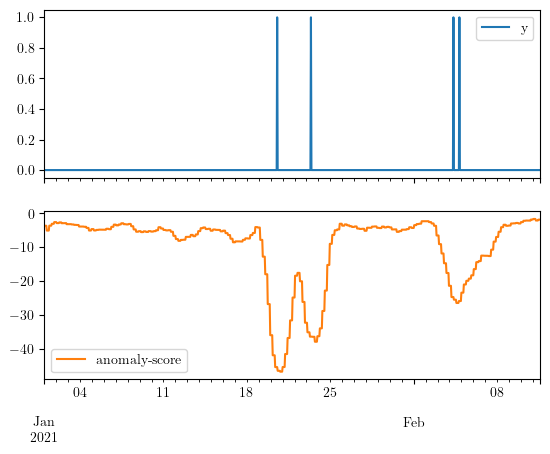}
	\end{subfigure}
	\caption{Left: the derived topological vectorization of the entire time series computed relative to each center according to Algorithm \ref{algo:vec}. Right: (top, in blue) the binary anomalous timestamps $y$ of the original signal, matches (bottom, in orange) the topological anomaly score based on the dimension 0 and 1 features of Algorithm \ref{algo:detection_score}.}
	\label{fig:sw-topolscore}
\end{figure}

Now, for a new vector $v$, a detection score is built via
\begin{align}\label{eq:detection_score}
s^2(v) = (v-\hat{\mu})^T \hat{\Sigma}^{-1}(v - \hat{\mu}),
\end{align}
that expresses the normalized distance to the mean behaviour of the base regime. We refer to an illustrative example in Figure \ref{fig:sw-topolscore}.

If we let $\hat{s}$ denote the score function based on data $(Y_t)_{t \in [0, T]}$, then anomaly detection tests of the form
\begin{align*}
	T_\alpha(v) = \1_{\hat{s}(v) \geqslant t_\alpha}
\end{align*}   
may be built. The relevance of this family of test based on $\hat{s}$ is assessed via the ROC\_AUC and RANGE\_PR\_AUC metrics of the Application Section \ref{sec:applications}, see there for more details.
Should a test with specific type I error $\alpha$ be needed, a calibration of $t_\alpha$ as the $1-\alpha$ quantile of scores on the base sample could be performed. Section \ref{sec:theo_test} theoretically proves that this strategy is grounded.

\subsection{Summary}

We can now summarize the whole procedure into the following algorithm, with complementary descriptive scheme in Figure \ref{fig:tada-pipeline}.

\begin{algorithm}[H]\label{algo:detection_score}
	\SetAlgoLined
	\KwIn{A window size $\Delta$, a dimension $K$. Possibly a stopping time $T$ or a mini-batch size $q$, and a ratio $h$ (corrupted observations).}
	\KwData{A multivariate time serie $(Y_t)_{t \in [0,L]} \in \mathbb{R}^D\times [0,L]$ (base regime, possibly corrupted)}
	
	Convert $(Y_t)_{t \in [0,L]}$ into $n$ persistence diagrams $(X_i)_{i=1, \hdots, n}$ via Algorithm~\ref{algo:timeserie_to_pers_diagrams}\;
	
	Build $K$ centroids $(c_1, \hdots, c_K)$ using Algorithm \ref{algo:kmeans_batch} or Algorithm \ref{algo:kmeans_minibatch}\;
	
	Convert $(X_i)_{i=1, \hdots, n}$ into $v_1, \hdots, v_n$ using Algorithm \ref{algo:vec}\;
	
	\KwOut{A score function $s: \R^K \rightarrow \R^+$, defined by \eqref{eq:detection_score}.}\caption{TADA: Detection score from base regime time series}
\end{algorithm}

Note that using Algorithm \ref{algo:kmeans_batch} with $T=2 \lceil \log(n) \rceil$ (see Theorem \ref{thm:cv_batch}) and assuming that base observations are not corrupted results in only two parameters specification for Algorithm \ref{algo:detection_score}: the window size $\Delta$ onto which correlation are computed, and $K$ the size of the vectorizations of persistence diagrams. At first it seems that choosing the right $K$ has the same amount of constraints than choosing the right $K$ for k-means imply, but it is actually made slightly easier by the fact that in this instance the k-means like procedure of Atol operates on the space of diagrams $\mathbb{R}^2$. As for $h$ it is very data-dependent, and optimizing for it is outside of the scope of this paper. In practice we will use a fixed contamination parameter of 0.1, a default fixed topological embedding size $K=10$, a default fixed number of restart for the k-means initialisations $n_{start} = 10$. As for the sliding window algorithm we use by default a window of size $\Delta = 100$ with a stride of $s=10$. The window size $\Delta$ is the parameter that most needs to adapt to the data or learning needs, see more on this in Section \ref{sec:applications} Applications.

\begin{figure}[h!]
	\centering
	\includegraphics[width=.8\linewidth]{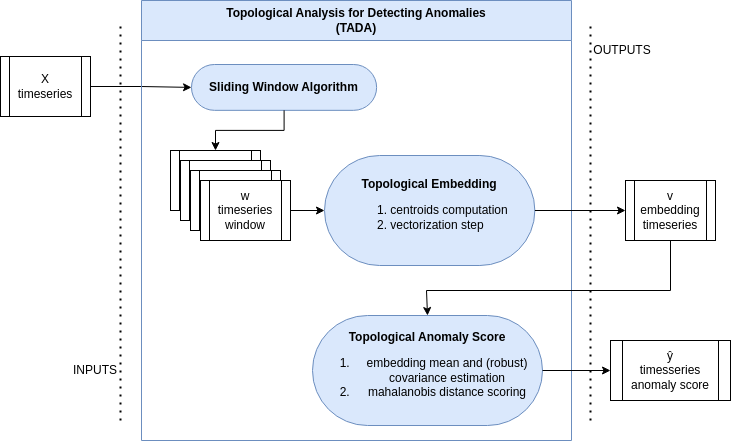}
	\caption{TADA general scheme for producing anomaly scores with topological information from the original timeseries.}
	\label{fig:tada-pipeline}
\end{figure}

Our proposed anomaly detection procedure Algorithm \ref{algo:detection_score} has a \textit{lean} design for the following reasons.
First it has very few parameters coming with default values.
Second, very little tuning is needed. Note that in the entire application sections to come, the only parameter to change will be the $\Delta$ window resolution parameter, a parameter shared with other methods.
Third, upon learning some data, TADA does not require a lot of memory: only the results of Algorithms \ref{algo:vec} (centroids) and \ref{algo:detection_score} (training vectorization mean and variance) are needed in order to produce topological anomaly scores. This implies that our methodology is easy to deploy, and requires no memory of training data which is often welcome in contexts of privacy for instance. It also means that the methodology will compare very favorably to methods that are memory-heavy such as tree-based methods, neural networks, etc.

\section{Theoretical results}\label{sec:theoretical_results}
In this section we intend to assess the relevance of our methodology from a theoretical point of view. Sections \ref{sec:cv_centroids} and \ref{sec:theo_test} gives results in the general case where the sample is a stationary sequence of random measures.  Section \ref{sec:theo_measures_PD} provides some details on how persistence diagrams built from mulitvariate time series as exposed in Section \ref{sec:methodo_persistence} can be casted into this general framework.
\subsection{Convergence of Algorithms \ref{algo:kmeans_batch} and \ref{algo:kmeans_minibatch}}\label{sec:cv_centroids}

In what follows we assume that $X_1, \hdots, X_n$ is a stationary sequence of random measures over $\R^d$, with common distribution $X$. Some assumptions on $X$ are needed to ensure convergence of Algorithms \ref{algo:kmeans_batch} and \ref{algo:kmeans_minibatch}.

First, let us introduce here  $\mathcal{M}_{N_{\max}}(R,M)$ as the set of random measures that are bounded in space, mass and support size.

\begin{defi}\label{def:bounded_measures}
For $R, M >0$ and $N_{\max} \in \N^*$, we let $\mathcal{M}_{N_{\max}}(R,M)$ denote the set of discrete measures on $\R^d$ that satisfies
\begin{enumerate}
\item $\supp(\mu) \subset \B(0,R)$,
\item $\mu(\R^2) \leqslant M$,
\item $|\supp(\mu)| \leqslant N_{\max}$.
\end{enumerate}
Accordingly, we let $\mathcal{M}(R,M)$ denote the set of measures such that $1$ and $2$ hold.
\end{defi}
With a slight abuse, if $X$ denote a distribution of random measures, we will write $X \in \mathcal{M}_{N_{\max}}(R,M)$ whenever $X \in \mathcal{M}_{N_{\max}}(R,M)$ almost surely. As detailed in Section \ref{sec:theo_measures_PD}, persistence diagrams built from correlation matrices satisfy the requirements of Definition \ref{def:bounded_measures}.  

In the i.i.d. case,  \cite{Chazal21} proves that the output of Algorithm \ref{algo:kmeans_batch} with $T= 2 \log(n)$ iterations returns a statistically optimal approximation of
\begin{align}\label{eq:defi_centres_opt}
\cb^* \in \mathcal{C}_{opt} = \argmin_{\cb \in (\R^2)^K} \E(X)(du) \min_{j=1, \hdots, K}\|u-c_j\|^2:= F(\cb),
\end{align}
where $\E(X)$ is the so-called mean measure $\E(X): A \in \mathcal{B}(\R^2) \mapsto \E(X(A))$. Note here that $X \in \mathcal{M}(R,M)$ ensures that $\mathcal{C}_{opt}$ is non-empty (see, e.g., \cite[Section 3]{Chazal21}). For the aforementioned result to hold, a structural condition on $\E(X)$ is also needed.

For a vector of centroids $\cb = (c_1, \hdots, c_k) \in \mathcal{B}(0,R)^k$, we let 
\begin{align*}
W_j(\cb)& \begin{multlined}[t] = \{ x \in \mathbb{R}^d \mid  \forall i < j \quad \| x- c_j\| < \| x-c_i\| \quad \mbox{and} \\
\forall{i >j} \quad \| x- c_j\| \leqslant \|x-c_i\| \}, 
\end{multlined}\\
N(\cb) & = \{x \mid \exists i < j \quad x \in W_i(\cb) \quad \mbox{and} \quad \|x-c_i\| = \|x-c_j\| \}, 
\end{align*}
so that $(W_1(\cb), \hdots, W_k(\cb))$ forms a partition of $\R^2$ and $N(\cb)$ represents the skeleton of the Voronoi diagram associated with $\cb$. The margin condition below requires that the mass of $\E(X)$ around $N(\cb^*)$ is controlled, for every possible optimal $\cb^* \in \mathcal{C}_{opt}$. To this aim, let us denote by $\B(A,t)$ the $t$-neighborhood of $A$, that is $\{ y \in \R^d \mid \d(y,A) \leqslant t \}$, for any $A \subset \R^d$ and $t\geqslant0$. The margin condition then writes as follows.

   \begin{defi}\label{def:margincondition}
         $\E(X) \in \mathcal{M}(R,M)$ satisfies a margin condition with radius $r_0 >0$ if and only if, for all $0 \leqslant t \leqslant r_0$,
					\begin{align*}
					 \sup_{\cb^* \in \mathcal{C}_{opt}} \E(X) \left ( \mathcal{B}(N(\cb^*),t) \right )  \leqslant \frac{B p_{min}}{128 R^2}t,
					\end{align*}
					where $\mathcal{B}(N(\cb^*),t)$ denotes the $t$-neighborhood of $N(\cb^*)$ and
					\begin{enumerate}
					\item $B=\inf_{\cb^* \in \mathcal{C}_{opt}, j \neq i}{\|c_i^* - c_j^*\|}$,
					\item $p_{min} = \inf_{\cb^* \in \mathcal{C}_{opt}, j=1, \hdots, k}{\E(X) \left ( W_j(\cb^*) \right )}$.
					\end{enumerate}
         \end{defi}
According to \cite[Proposition 7]{Chazal21}, $B$ and $p_{\min}$ are positive quantities whenever $\E(X) \in \mathcal{M}(R,M)$. In a nutshell, a margin condition ensures that the mean distribution $\E(X)$ is well-concentrated around $k$ poles. For instance, finitely-supported distributions satisfy a margin condition. 
Up to our knowledge, margin-like conditions are always required to guarantee convergence of Lloyd-type algorithms \cite{Monteleoni16,Levrard18} in the i.i.d. case.

Turning back to our base case of time series of persistence diagrams, we cannot assume anymore independence between observations. To adapt the argument of \cite{Chazal21} in our framework, a quantification of dependence between discrete measures is needed. We choose here to seize dependence between observation via $\beta$-mixing coefficients, whose definition is recalled below.

\begin{defi}\label{def:beta_mixing}
For $t \in \Z$ we denote by $\sigma(- \infty, t)$ (resp. $\sigma(t,+ \infty)$) the $sigma$-fields generated by $ \hdots, X_{t-1}, X_t$ (resp. $X_t, X_{t+1} \hdots $). The \textit{beta-mixing} coefficient of order $q$ is then defined by
\begin{align*}
\beta(q)=\sup_{t \in \Z} \E \left [ \sup_{B \in \sigma(t+q,+ \infty)} | \P(B\mid \sigma(- \infty, t)) - \P(B) | \right ].
\end{align*}
\end{defi} 
Recalling that the sequence of persistence diagrams is assumed to be stationary, its beta-mixing coefficient of order $q$ may be subsequently written as
\begin{align*}
\beta(q) = \E (\d_{TV}(P_{(X_q, X_{q+1}, \hdots)  \mid \sigma(\hdots, X_0)},P_{(X_q, X_{q+1}, \hdots)})),
\end{align*}
where $\d_{TV}$ denotes the total variation distance and $P_Z$ denotes the distribution of $Z$, for a generic random variable $Z$. As detailed in Section \ref{sec:theo_measures_PD}, mixing coefficients of persistence diagrams built from a multivariate time serie may be bounded in terms of mixing coefficients of the base time serie. Whenever these coefficients are controlled, results from the i.i.d. case may be adjusted to the dependent one.

We begin with an adaptation of \cite[Theorem 9]{Chazal21} to the dependent case.

\begin{theorem}\label{thm:cv_batch}
Assume that $X_1, \hdots, X_n$ is stationary, with distribution $X \in \mathcal{M}_{N_{max}}(R,M)$, for some $N_{max} \in \mathbb{N}^*$. Assume that $\mathbb{E}(X)$ satisfies a margin condition with radius $r_0$, and denote by $R_0= \frac{Br_0}{16\sqrt{2}R}$, $\kappa_0 = \frac{R_0}{R}$. For $q \in \N^*$, choose $T \geqslant  \lceil \frac{\log(n/q)}{\log(4/3)} \rceil$, and let $\cb^{(T)}$ denote the output of Algorithm \ref{algo:kmeans_batch}. 

If $q$ is such that ${\beta(q)^2}/{q^3} \leqslant  n^{-3}$, and $\cb^{(0)} \in \B(\mathcal{C}_{opt},R_0)$, then, for $n$ large enough, with probability larger than $1 - c \frac{q k M^2}{n\kappa_0^2 p_{\min}^2}-2e^{-x}$, we have
\begin{align*}
\inf_{\cb^* \in \mathcal{C}_{opt}} \| \cb^{(T)} - \cb^*\|^2  \leqslant \frac{ B^2r_0^2}{512 R^2 } \left ( \frac{q}{n} \right ) + C \frac{M^2 R^2 k^2 d \log(k)}{p_{min}^2} \left ( \frac{q}{n} \right )(1+x),
\end{align*}
for all $x>0$, where $C$ is a constant. 

Moreover, if $q$ is such that $\beta(q)/q \leqslant n^{-1}$ and $\cb^{(0)} \in \B(\cb^*,R_0)$, it holds
\begin{align*}
\E \left ( \inf_{\cb^* \in \mathcal{C}_{opt}} \| \cb^{(T)} - \cb^*\|^2 \right ) \leqslant C \frac{d k^2 R^2 M^2  \log(k)}{\kappa_0^2 p_{\min}^2} \left ( \frac{q}{n} \right )
\end{align*}
\end{theorem}

Intuitively speaking, Theorem \ref{thm:cv_batch} provides the same guarantees as in the i.i.d. case, but for a 'useful' sample size $n/q$ that corresponds to the number of sample measures that are spaced enough (in fact $q$-spaced) so that they are independent enough in view of the targeted convergence rate in $q/n$. This point of view seems ubiquitous in machine learning results based on dependent sample (see, e.g., \cite[Theorem 1]{Agarwal13} or \cite[Lemma 7]{Mohri10}). 

Assessing the optimality of the requirements on $\beta(q)$ seems a difficult question. Following  \cite[Corollary 20]{Mohri10} and comments below, the $\beta(q) \leqslant q/n$ condition we require to get a convergence rate in expectation seems optimal for polynomial decays ($\beta(q) = O(q^{-a})$, $a>0$) in an empirical risk minimization framework. However, this choice leads to a convergence rate in $(q/n)^{(a-1)/(4a)}$ for \cite{Mohri10}, larger than our $(q/n)$ rate. Though  the output of Algorithm \ref{algo:kmeans_batch} is not an empirical risk minimizer, it is likely that it has the same convergence rate as if it were (based on a similar behavior for the plain $k$-means case, see e.g., \cite{Levrard18}). The difference between convergence rates given in \cite{Mohri10} and Theorem \ref{thm:cv_batch} might be due to the fact that \cite{Mohri10} settles in a 'slow rate' framework, where the convexity of the excess risk function is not leveraged, whereas a local convexity result is a key argument in our result (explicited by \cite[Lemma 21]{Chazal21}).

In a fast rate setting (i.e. when the risk function is strictly convex), \cite[Theorem 5]{Agarwal13} also suggests that a milder requirement in  $\beta(q)/q \leqslant n^{-1}$ 
might be enough to get a $O(q/n)$ convergence rate in expectation, for online algorithms under some assumptions that will be discussed below Theorem \ref{thm:cv_minibatch} (convergence rates for an online version of Algorithm \ref{algo:kmeans_batch}). Up to our knowledge there is no lower bound in the case of stationary sequences with controlled $\beta$ coefficients that could back theoretical optimality of such procedures.

At last, the sub-exponential rate we obtain in the deviation bound under the stronger condition $\beta(q)^2/q^3 \leqslant n^{-3}$ seems better than the results proposed in \cite[Corollary 20] {Mohri10} or \cite[Theorem 5]{Agarwal13} in terms of 'large deviations' (in $(q/n)x$ here to get an exponential decay). Determining whether the same kind of result may hold under the condition $\beta(q) \leqslant q/n$ remains an open question, as far as we know.

Nonetheless, Theorem \ref{thm:cv_batch} provides some  convergence rates (in expectation) for several decay scenarii on $\beta(q)$:
\begin{itemize}
\item if $\beta(q) \leqslant C \rho^q$, for $\rho < 1$, then an optimal choice of $q$ is $q = c \log(n)$, providing the same convergence rate as in the i.i.d. case \cite[Theorem 9]{Chazal21}, up to  a $\log(n)$ factor.
\item if $\beta(q) = C q^{-a}$, for $a >0$, then an optimal choice of $q$ is $q= C n^{\frac{1}{a+1}}$, that yields a convergence rate in $n^{-1 + \frac{1}{a+1}}$. 
\end{itemize}
In the last case, letting $a \rightarrow + \infty$ allows to retrieve the i.i.d. case, whereas $a \rightarrow 0$ has for limiting case the framework when only one sample is observed (thus leading to a non-learning situation). 

Whatever the situation, a benefit of Algorithm \ref{algo:kmeans_batch} is that a correct choice of $q$, thus the prior knowledge of $\beta(q)$, is not required to get an at least consistent set of centroids, by choosing $T =  \lceil \frac{\log(n)}{\log(4/3)} \rceil$. This will not be the case for the convergence of Algorithm \ref{algo:kmeans_minibatch}, where the size of minibatches $q$ is driven by a prior knowledge on $\beta$.

\begin{theorem}\label{thm:cv_minibatch}
Let $q$ be large enough so that $\frac{\beta(q/18)}{q^2} \leqslant n^{-2}$ and $q \geqslant c_0 \frac{k^2 M^2}{p_{\min}^2 \kappa_0^2} \log(n)$, for a constant $c_0$ that only depends on $\int_0^1 \beta^{-1}(u)du$. Provided that $\E(X)$ satisfies a margin condition, if the initialization satisfies the same requirements as in Theorem \ref{thm:cv_batch}, then the output of Algorithm \ref{algo:kmeans_minibatch} satisfies
\begin{align*}
\E \left ( \inf_{\cb^* \in \mathcal{C}_{opt}} \| \cb^{(T)} - \cb^*\|^2 \right ) \leqslant 128 \frac{k d M R^2}{p_{\min} (n/q)}.
\end{align*}
\end{theorem}
As in \cite{Rio93}, the generalized inverse $\beta^{-1}$ is defined by $\beta^{-1}(u) = \left |  \left \{k \in \N^* \mid \beta(k) > u \right \} \right |$. In particular, for $\beta(q) \sim q^{-a}$, $\int_{0}^1 \beta^{-1}(u) du$ is finite only if $a > 1$ (that precludes the asymptotic $a \rightarrow 0$). 

The requirement $\beta(q)/q^2 = O(n^{-2})$ is stronger than in Theorem \ref{thm:cv_batch}, thus stronger that the $\beta(q)/q = O(n^{-1})$ suggested by \cite[Theorem 5]{Agarwal13} in a similar online setting. Note however that for \cite[Theorem 5]{Agarwal13} to provide a $O(q/n)$ rate under the requirement $\beta(q)/q = O(n^{-1})$, two other terms have to be controlled:
\begin{enumerate}
\item a total step sizes term in $\sum_{t=1}^T \|\cb^{t} - \cb^{(t-1)}\|$ that must be of order $O(1)$. Controlling this term would require an slight adaptation of Algorithm \ref{algo:kmeans_minibatch}, for instance by clipping gradients.
\item a regret term in $\E \left ( \sum_{t=1}^T \bar{X}_{B_T}(du)\left [ \d^2(u,\cb^{(t)}) - \d^2(u,\cb^*) \right ] \right )$ that must be of order $O(q)$. The behavior of this term remains unknown in our setting, so that determining whether the milder condition $\beta(q)/q=O(n^{-1})$ is sufficient remains an open question.
\end{enumerate}

Let us emphasize here that, to optimize the bound in Theorem \ref{thm:cv_minibatch}, that is to choose the smallest possible $q$, a prior knowledge of $\beta(q)$ is required. This can be the case when the original multivariate time serie $Y_t$ follows a recursive equation as in \cite{Bourakna22}. Otherwise, these coefficients may be estimated, using histograms as in \cite{McDonald15} for instance.

As in the batch case, the required lower bound on $q$ corresponds to the "optimal" choice of minibatch spacings so that consecutive even minibatches may be considered as i.i.d.. It is then not a surprise that we recover the same rate as in the i.i.d. case, but with $n/q$ samples (see \cite[Theorem 10]{Chazal21}). As for the batch situation, several decay scenarii may be considered:
\begin{itemize}
\item for $\beta(q) \leqslant C \rho^q$, $\rho <1$, choosing $q = c_0 \frac{k^2 M^2}{p_{\min}^2 \kappa_0^2} \log(n)$ for a large enough $c_0$ is enough to satisfy the requirements of Theorem \ref{thm:cv_minibatch}, and yields the same result as in the i.i.d. case (\cite[Theorem 10]{Chazal21}).
\item for $\beta(q) = C q^{-a}$, $\beta^{-1}(u) = (C/u)^{1/a}$. An optimal choice for $q$ is then $C n^{\frac{2}{a+2}}$, leading to a convergence rate in $n^{-1 + \frac{2}{a+2}}$. 
\end{itemize}
Let us mention here that the stronger condition in Theorem \ref{thm:cv_minibatch} leads to a slower convergence bound for the polynomial decay case, compared to the output of Algorithm \ref{algo:kmeans_batch}. Again, assessing optimality of exposed convergence rates remains an open question, up to our knowledge.

\subsection{Test with controlled type I error rate}\label{sec:theo_test}

In this section, we investigate the type I error of the test
\begin{align*}
T_\alpha \mapsto \1_{s(v) \geqslant t_{n,\alpha}},
\end{align*}
where $s$ is the score function built in Section \ref{sec:anomaly_detection} and $t_{n,\alpha}$ will be built from sample to achieve a type I error rate below $\alpha$. 

To keep it simple, we assume that $\Sigma$ and $\mu$ in \eqref{eq:mean_cov_vectorizations} are computed from a separate sample, so that we observe 
\begin{align*}
\tilde{v}_i=\Sigma^{-1/2}(v_i-\mu)
\end{align*}
from a stationary sequence of measures, resulting in a stationary sequence of vectors. Whether $\Sigma$ and $\mu$ should be computed on the same sample, extra terms involving concentration of $\Sigma$ and $\mu$ around their expectations should be added, as in the i.i.d. case.

We let $Z$ denote the common distribution of the $s(v_i)=\|\tilde{v}_i\|$'s,  that represent the "normal" behavior distribution of the time serie structure. For the test $T_\alpha$ introduced above, its type I error is then 
\begin{align*}
\P_{\|\tilde{v}\| \sim Z}(\|\tilde{v}\| > t_{n,\alpha}).
\end{align*}
A common strategy here is to choose a $t_{n,\alpha}$ from sample, such as
\begin{align*}
\frac{1}{n} \sum_{i=1}^n \1_{\|\tilde{v}_i\| > t_{n,\alpha}} \leqslant \alpha - \delta,
\end{align*}
for a suitable $\delta < \alpha$. In what follows we denote by $\hat{t}$ such an empirical choice of threshold. 
The following result ensures that this natural strategy remains valid in a dependent framework.

\begin{prop}\label{prop:test_beta_mixing}
Let $q \in [\![1,n]\!]$, and $\alpha$, $\delta$ be positive quantites that satisfy
\begin{align*}
5\sqrt{\alpha}\sqrt{\frac{\log(n)}{(n/q)}} \leqslant \delta < \alpha.
\end{align*}
If $\hat{t}$ is chosen such that 
\begin{align*}
\frac{1}{n} \sum_{i=1}^n \1_{\|\tilde{v}_i\| > \hat{t}} \leqslant \alpha - \delta,
\end{align*}
then, with probability larger than $1-\frac{4q}{n}- \beta(q) \sqrt{\frac{n}{\alpha q}}$, it holds
\begin{align*}
\P_{\|\tilde{v}\| \sim Y}(\|\tilde{v}\| > \hat{t}) \leqslant \alpha.
\end{align*}
\end{prop}
In other words, Proposition \ref{prop:test_beta_mixing} ensures that the anomaly detection test $\1_{\|\tilde{v}\| \geqslant \hat{t}}$ has a type I error below $\alpha$, with high probability. Roughly, this bound ensures that, for confidence levels $\alpha$ above the statistical uncertainty of order $q/n$, tests with the prescribed confidence level may be achieved via increasing the threshold by a term of order $\sqrt{\alpha q/n}$. 

As for Theorem \ref{thm:cv_batch} or \ref{thm:cv_minibatch}, choosing the smaller $q$ that achieves $\beta(q)^2/q^3 \leqslant \alpha/n^3$ optimizes the probability bound in Proposition \ref{prop:test_beta_mixing}:
\begin{itemize}
\item for $\beta(q) \leqslant C \rho^q$, $q$ of order $C' \log(n)$ is enough to satisfy $\beta(q)/q^2 \leqslant n^{-2}$, providing the same results as in the i.i.d. case, up to a $\log(n)$ factor.
\item for $\beta(q) = Cq^{-a}$, an optimal choice for $q$ is of order $ n^{\frac{3}{2a+3}}\alpha^{-1/(2a+3)}$, that leads to the same bounds as in the i.i.d. case, but with useful sample size $n/q = n^{\frac{2a}{2a+3}} \alpha^{1/(2a+3)}$.
\end{itemize} 

Although this bound might be sub-optimal, it can provides some pessimistic prescriptions to select a threshold $\alpha - \delta$, provided that the useful sample size $n/q$ is known. For instance, assuming $\log(n) \leqslant 6$, for $\alpha = 5 \%$ the minimal $\delta$ is of order $2.7/\sqrt{n/q}$, that is neglictible with respect to $\alpha$ whenever $n/q$ is large compared to roughly $3000$.

\subsection{Theoretical guarantees for persistence diagrams built from multivariate time series}\label{sec:theo_measures_PD}

We discuss here how the outputs $(X_t)_{t=1, \hdots n}$ of Algorithm \ref{algo:timeserie_to_pers_diagrams}, based on a $D$-dimensional time serie $(Y_t)_{t\in [0;L]}$, fall in the scope of the previous sections. Recall that a persistence diagram may be thought of as a discrete measure on $\R^2$ (see Section \ref{sec:centroids}). In a nutshell, if $(Y_t)_{t\in [0;L]}$ is stationary with a certain profile of mixing coefficients, then so would $(X_t)_{t=1, \hdots n}$.

\smallskip

\noindent \textit{Stationarity}: Since, for $t=1, \hdots, n$, $X_t$ may be expressed as $f((Y_u)_{u \in [s(t-1);s(t-1) + \Delta]}$, for some function $f$, then stationarity of $(X_t)_{t=1, \hdots, n}$ follows from stationarity of $(Y_t)_{t \in [0;L]}$.

\medskip

\noindent \textit{Boundedness}: We intend to prove that the outputs of Algorithm \ref{algo:timeserie_to_pers_diagrams} are in the scope of Definition \ref{def:bounded_measures}. Let $d$ be an homology dimension, $t \in [\![1;n]\!]$, and recall that $X_t$ is the order $d$ persistence diagram built from the Vietoris-Rips filtration of $([\![1;D]\!],E,S_t)$, where $E$ is the set of all edges and $S_t=1-Corr(Y_{[s(t-1);s(t-1) + \Delta]})$ gives the weights that are filtered (see Section \ref{sec:methodo_persistence}). First note that, for every $1 \leqslant i,j \leqslant D$, $S_{t, (i,j)} \in [0;2]$, so that every point in the persistence diagram is in $[0;2]^2$.  Next, since a birth of a $d$-order feature is implied by an addition of a $d$-order simplex in the filtration (see for instance \cite{boissonnat2018geometric} Section 11.5, Algorithm 11), the total number of points in the diagram is bounded by $\binom{D}{d+1}$. At last, for a bounded weight function $\omega$, the total mass of $X_t$ may be bounded by $\binom{D}{d+1} \times \| \omega \|_\infty$. We deduce that $X_t \in \mathcal{M}_{N_{\max}}(R,M)$ (Definition \ref{def:bounded_measures}), with $R \leqslant 4$, $N_{\max} \leqslant \binom{D}{d+1}$, and $M \leqslant \binom{D}{d+1} \times \| \omega \|_\infty$. Note that in the experiments, we set $\omega \equiv 1$.

\medskip

\noindent \textit{Mixing coefficients}: Here we expose  how the mixing coefficients of $(X_t)_{t =1, \hdots, n}$ (Definition \ref{def:beta_mixing}) may be bounded in terms of those of $(Y_t)_{t \in [0;L]}$. Let us denote these coefficients by $\beta$ and $\tilde{\beta}$. If the stride $s$ is larger than the window size $\Delta$, then it is immediate that, for all $q\geqslant 1$, $\beta(q) \leqslant \tilde{\beta}(qs-\Delta)$. If the stride $s$ is smaller (or equal) than $\Delta$, then, denoting by $q_0 = \lfloor (\Delta/s) \rfloor +1$, we have, for $q < q_0$, $\beta(q) \leqslant 1$ (overlapping windows), and, for $q \geqslant q_0$, $\beta(q) \leqslant \tilde{\beta}(qs - \Delta)$. The mixing coefficients of $X_t$ may thus be controlled in terms of those of $Y_t$. For fixed $\Delta$ and $s$, this ensures that mixing coefficients of $X_t$ and $Y_t$ have the same profile (and leads to the same convergence rates in Theorem \ref{thm:cv_batch} and \ref{thm:cv_minibatch}).
\begin{itemize}
\item If $\tilde{\beta}(q) \leqslant C_Y q^{-a}$, for $C_Y,a > 0$, then, for any $q \geqslant q_0$, $\beta(q) \leqslant C_Y \left ( s - \Delta/q_0 \right )^{-a} q^{-1}$, so that $\beta(q) \leqslant C_X q^{-a}$, for some constant $C_X$ (depending on $q_0$, $s$ and $a$).
\item If $\tilde{\beta}(q) \leqslant C_Y \tilde{\rho}^q$, for $C_Y >0$ and $\tilde{\rho} < 1$, then, for any $q \geqslant q_0$, $\beta(q) \leqslant C_Y(\tilde{\rho}^{(s-(\Delta/q_0))})^q$, so that $\beta(q) \leqslant C_X \rho^q$, for some $C_X >0$ and $\rho < 1$ depending on $q_0$, $s$ and $\tilde{\rho}$.
\end{itemize} 
In turn, mixing coefficients of $Y_t$ may be known or bounded, for instance in the case where it follows a recursive equation (see, e.g., \cite[Theorem 3.1]{Pham85}), or inferred (see, e.g., \cite{McDonald15}). Interestingly, the topological wheels example provided in Section \ref{sec:top_wheels} (borrowed from \cite{Bourakna22}) falls into the sub-exponential decay case.

\medskip

\textit{Margin condition}: The only point that cannot be theoretically assessed in general for the outputs of Algorithm \ref{algo:timeserie_to_pers_diagrams} is to know whether $\E(X)$ satisfies the margin condition exposed in Definition \ref{def:margincondition}. As explained below Definition \ref{def:margincondition}, a margin condition holds whenever $\E(X)$ is concentrated enough around $k$ poles. Thus, structural assumptions on $1-Corr(Y_{[0;\Delta]})$ (for instance $k$ prominents loops) might entail $\E(X)$ to fulfill the desired assumptions (as in \cite{Levrard15} for Gaussian mixtures). However, we strongly believe that the requirements of Definition \ref{def:margincondition} are too strong, and that convergence of Algorithms \ref{algo:kmeans_batch} and \ref{algo:kmeans_minibatch} may be assessed with milder smoothness assumptions on $\E(X)$. This fall beyond the scope of this paper, and is left for future work. The experimental Section \ref{sec:applications} to follow assesses the validity of our algorithms in practice.

\section{Applications}\label{sec:applications}

In order to make the case for the efficiency of our proposed anomaly detection procedure TADA, we now present an assortment of both real-case and synthetic applications.
The first application we call the {Topological Wheels} problem that is directly derived from \cite{Bourakna22} to show the relevance of a topologically based anomaly detection procedure on complex synthetic data that is designed to mimic dependence patterns in brain signals.
The second application is an up-to-date replication of a benchmark with the {TimeEval library} from \cite{SchmidlEtAl2022Anomaly} on a large array of synthetic datasets to quantitatively demonstrate competitiveness of the proposed procedure with current state-of-the-art methods.
The third application is a real-case dataset from \cite{exathlon} consisting in data traces from repeated executions of large-scale stream processing jobs on an Apach Spark cluster.
Lastly we produce interpretability elements for the anomaly detection procedure TADA.

Evaluation of an anomaly detection procedure in the context of time series data has many pitfalls and can be hard to navigate, we refer to the survey of \cite{sorbo23}. Here we mainly evaluate anomaly scores with the robustified version of the Area Under the PR-Curve: the Range PR\_AUC metric of \cite{Paparrizos2022} (later just "RANGE\_PR\_AUC"), where a metric of 1 indicates a perfect anomaly score, and a metric close to 0 indicates that the anomaly score simply does not point to the anomalies in the dataset. For the sake of comparison with the literature we also include the Area Under the ROC-Curve metric (later just "ROC\_AUC") although as \cite{sorbo23} demonstrated it is simply less accurate and powerful a metric in the unbalanced context of anomaly detection.
Therefore each collection of anomaly detection problems will yield evaluation statistics. To summarize comparisons between algorithms we use a critical difference diagram, that is a statistical test between paired populations using the package \cite{Herbold2020}. 
We introduce two other statistical summary of interest:
\begin{itemize}
	\item the "$\# >.9$" metric, which we introduce for the number of anomaly detection problems an algorithm has a RANGE\_PR\_AUC over .9, which we roughly translates as "finding" the anomalies in the dataset or "solving" the problem,
	\item the "$\# \text{rank}1$" metric, which we introduce for the number of problems where an algorithm reaches the best PR\_AUC score over other algorithms. If a method reaches a not negligible number, this indicates that it makes sense to use the method for solving this kind of problem.
\end{itemize}

For the purpose of comparison with the state-of-the-art we draw from the recent benchmark of \cite{SchmidlEtAl2022Anomaly}. We take the \textit{three best performing} methods from the unsupervised, multivariate category: the "KMeansAD" anomaly detection based on k-means cluster centers distance using ideas from \cite{Yairi2001FaultDB}, the baseline density estimator k-nearest-neighbours algorithm on time series subsequences "SubKNN", and "TorskAD" from \cite{Heim2019AdaptiveAD}, a modified echo state network for anomaly detection.
In order to understand better the value of the introduced topological methodology, we also couple the topological features of Algorithm \ref{algo:vec} to the isolation forest algorithm from \cite{IsoForest08} for an \textit{unsupervised} anomaly detection method denominated as "Atol-IF" in reference to the \cite{Royer19} paper.
For an upper bound on what can be achieved on the first collection of problem we also couple those topological features to a random forest classifier \cite{Breiman01}, resulting in a \textit{supervised} anomaly detection method denominated as "Atol-RF".
Lastly for discriminating effects of a spectral analysis respective to a topological analysis, we compute spectral features on the correlation graphs coupled to either an isolation forest or to a random forest classifier, in an \textit{unsupervised} anomaly detection method denominated as "Spectral-IF" and a \textit{supervised} one named "Spectral-RF".

In practice all those methods involve a form of time-delay-embedding or subsequence analysis or context window analysis (we use these terms synonymously in this work), that requires to compute a prediction from a window size number $\Delta$ of past observations. $\Delta$ is a key value that acts as the equivalent of image resolution or scale in the domain of time series.
In using a subsequence analysis, given a $\Delta$-uplet of timestamps $[t] := (t_1, t_2, ..., t_{\Delta})$, once an anomaly score $s_{[t]}$ is produced it is related to that particular $\Delta$-uplet but does not refer to a specific timestep. A window reversing step is needed to map the scores to the original timestamps.
For fair comparison, we will provide all methods with the following (same) last-step window reversing procedure: for every timestep $t$, one computes the sum of windows containing this timestep $\hat{s}_t := \sum_{[t']: t \in [t']} s_{[t']}$. Here we select not to use the more classical average $\tilde{s}_t := \sum_{[t']: t \in [t']} s_{[t']} / \sum_{[t']: t \in [t']} 1$, as this average produces undesirable border effects because the timestamps at the beginning and end of the signal are contained in less windows, which in turn makes them over-meaningful after averaging. Using the sum instead has no effect on anomaly scoring as the metrics are scale-invariant.

For the specific use of TADA in this section, the centroids computation part of Section \ref{sec:centroids} is made using $\omega_{(b,d)} = 1$ and computed with the batch version described in Algorithm \ref{algo:kmeans_batch}.
Our implementation relies on \cite{gudhi} for the topological data analysis part, but also makes use of the \cite{sklearn} Scikit-learn library for the anomaly detection part, minimum covariance determinant estimation and overall pipelining. The code is published as part of the ConfianceAI program \url{https://catalog.confiance.ai/} and can be found in the catalog: \url{https://catalog.confiance.ai/records/4fx8n-6t612}. For now its access is restricted but it will become open-source in the following months.
All computations are possible and were in effect made on a standard laptop (i5-7440HQ 2.80 GHz CPU).

\subsection{Introducing the {Topological Wheels} dataset.}\label{sec:top_wheels}

In this first application section we introduce a hard, multiple time series unsupervised problem that emulates brain functions, and then solve the problem with our proposed method and compare solutions with state-of-the-art anomaly detection methods as well as supervised concurrent methods.

\cite{Bourakna22} introduces ideas for evaluating methodologies relying on TDA such as ours. They allow to produce a multiple time series with a given node dependence structure from a mixture of latent autoregressive process of order two (AR2). One direct application for this type of data generation is to emulate the network structure of the brain, whose normal connectivity is affected by conditions such as ADHD or Alzheimer's disease. Therefore in accordance with \cite{Bourakna22} we design and introduce the {Topological Wheels} problem: a multiple time series datasets with a latent dependence structure "type I" as a normal mode, and sometimes a "type II" latent dependence structure as an abnormal mode. For the type I dependence structure we use a single wheel where every node are connected by pair, and every pair are connected to two others forming a wheel; then we connect a pair of pairs forming an 8 shape or double wheel, see Figure \ref{fig:tw-connect}. And for the type II structure we start from a double wheel and add another connection between two pairs. The first mode of dependence is the prominent mode for the timeseries duration, and is replaced for a short period at a random time by the second mode of dependence. The total signal involves 64 timeseries sampled at 500 Hz for a duration of 20 seconds, see Figure \ref{fig:tw-connect}. We produce ten such datasets and call them the {Topological Wheels} problem. It consists in being able to detect the change in underlying pattern without supervision. We note that by design the two modes are similar in their spectral profile, so overall detecting anomalies should be hard for methods that do not capture the overall topology of the dependency structure. The datasets are available through the public \cite{gudhi} library at \url{github.com/GUDHI/gudhi-data}.

\begin{figure}[h!]
	\centering
	\begin{subfigure}[m]{.49\linewidth}
		\includegraphics[width=\linewidth]{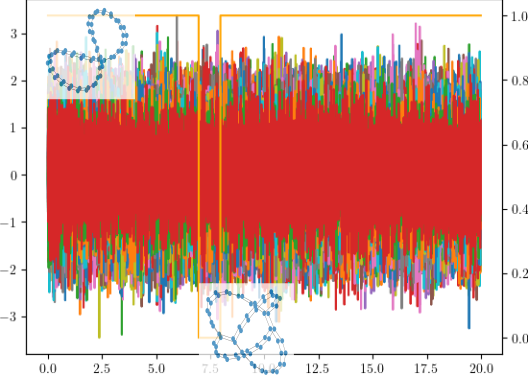}
	\end{subfigure}
	\begin{subfigure}[m]{.5\linewidth}
	\centering
	\includegraphics[width=\linewidth]{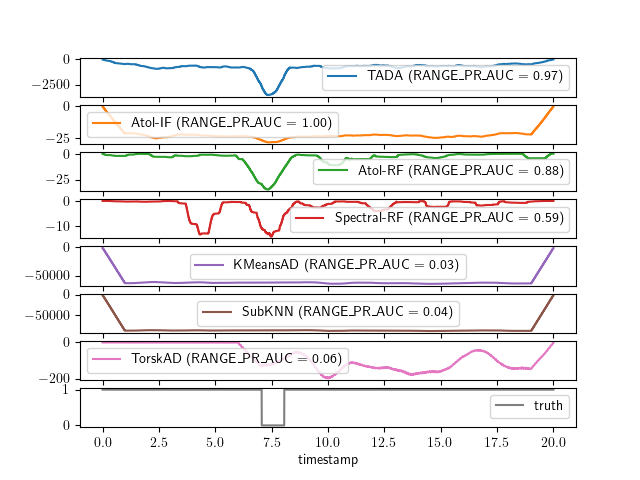}
	\end{subfigure}
	\caption{Left: Synthetised time series and the latent generating process (orange) indicating normal connexion (double circular wheel with middle connexion, on top) or abnormal connexion (double circular wheel with two connexions, on the bottom). Right: Anomaly scores of all tested methods on one of the datasets, and their RANGE\_PR\_AUC metric on comparing with the truth (bottom row) in parenthesis.}
	\label{fig:tw-connect}
\end{figure}

For learning this problem we use a cross-validation like procedure with a focus on evaluation: we perform ten experiments and for each experiment, every method is fitted on one dataset and evaluated on the other nine datasets. We then rotate the training dataset until all ten datasets have been used for training.
We use this particular setup in order to be able to compare supervised and unsupervised methods on comparable grounds.
As for method calibration we note the following: all methods are given (and use) the real length of the anomalous segment of $\Delta=500$ consecutive timestamps, and since all of them use window subsequences we set the same stride for all to be $s=10$. Lastly for all methods that use a fixed-size embedding ({Atol}-based methods and "{Spectral-RF}") we set the same size of the support to be $K=10$.

We first show in Figure \ref{fig:tw-connect} the results of one iteration of learning, that is when all methods are trained on a topological wheels dataset and evaluated on another.
The last row of the figure with label "truth" shows the underlying signal value of the evaluated dataset. The other rows are the computed anomaly score of each method along the time x-axis, with the convention that the lower the score, the more abnormal the signal. The corresponding RANGE\_PR\_AUC score of each method is written in the label.
This first example confirms the intuition that methods that do not rely on topology, that is the spectral method, the k-nearest-neighbour method and the modified echo state network method all fail to capture the anomaly. This is particularly striking for the spectral method as it was trained with supervision.
On the other hand all methods based on the topological features manage to capture some indication that there is anomaly in the signal. For the isolation forest method, even though it clearly separates the anomalous segment from the rest, it is not reliable as it seems to indicate other anomalies when there aren't. The random forest supervised method perfectly discriminates the anomalous segment from the rest of the time series, and so does our method almost as reliably.

\begin{figure}[h]
	\centering
	\begin{subfigure}[m]{.7\linewidth}
		\includegraphics[width=\linewidth]{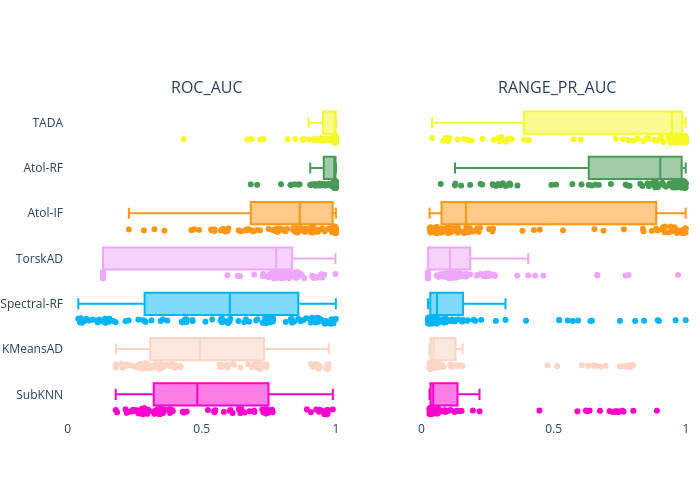}
	\end{subfigure}
	\begin{subfigure}[m]{.28\linewidth}
		\includegraphics[angle=90, width=\linewidth]{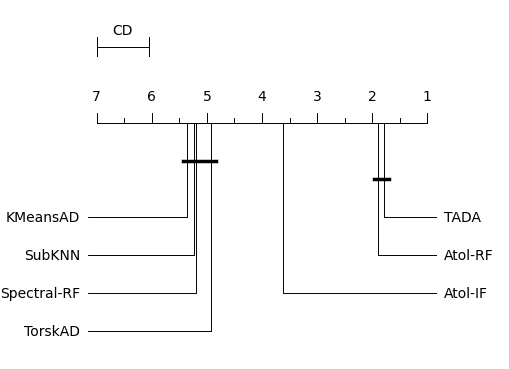}
	\end{subfigure}
	\caption{Left: aggregated results for the Topological Wheels problem in the form of box plots for the ROC\_AUC and RANGE\_PR\_AUC metrics, and below them the points the metrics have been computed from - where each point represent a metric score from comparing an anomaly score to the underlying truth. Right: autorank summary (see \cite{Herbold2020}) ranking of methods on the Toplogical Wheels problem, showing competitiveness between our unsupervised TADA method and the equivalent topological supervised method.}
	\label{fig:tw-results}
\end{figure}

\begin{table}[h]
\begin{center}
	\begin{tabular}{llrrrrr}
\toprule
  algorithm &  &  \#xp &  \#$>$.9 &  \#rank1 &  med time &  iqr time \\
\midrule
    Atol-IF & (unsupervised) &   90 &    21 &       7 &    21.070 &     0.099 \\
    Atol-RF & (supervised) &   90 &    45 &      38 &    21.104 &     0.104 \\
   KMeansAD & (unsupervised)&   90 &     0 &       0 &     0.801 &     0.110 \\
Spectral-RF & (supervised) &   90 &     2 &       3 &     4.737 &     0.022 \\
     SubKNN & (unsupervised)&   90 &     0 &       0 &    35.287 &     0.245 \\
       TADA & (unsupervised)&   90 &    54 &      40 &    21.128 &     0.115 \\
    TorskAD & (unsupervised)&   90 &     1 &       2 &   112.841 &     1.973 \\
\bottomrule
\end{tabular}

	\caption{Summary statistics on the Topological Wheels problem for the algorithms evaluated. All methods could produce scores for the 90 experiments, and without surprise the methods relying on topological analysis are overwhelmingly dominating other methods in 88 of 90 experiments. Our unsupervised method TADA is on par with a supervised learning method for the number of problem where it has the best PR\_AUC score  ($\#\text{rank}1$ column). In seconds, the computation median time ("med time") and interquartile range ("iqr time") are standard with respect to the data sizes - also note that computations are not optimized, and in fact performed on a single laptop.}
	\label{tab:tw-times}
\end{center}
\end{table}

We now look at the aggregated results for the entire problem, see Figure \ref{fig:tw-results} and times Table \ref{tab:tw-times}.
We present both ROC\_AUC and RANGE\_PR\_AUC averages with their standard deviations over experiments, as well as the computation times for the sake of completeness.
Neither the spectral procedure nor the echo state network, subKNN or k-nearest-neighbour method are able to capture any information from the Topological Wheels problem.
Using topological features with an isolation forest yields competitive results but it is simply inferior to our procedure. This demonstrates that the key information to this problem lies in the topological embedding which is not surprising, by design.
Our procedure solves this problem almost perfectly, and although it is unsupervised it is as competitive as a comparable supervised method.
This experiment demonstrates the impact of topology-based methods for anomaly detection, as the non-topology method fail to capture any of the signal in the datasets. Our proposed {TADA} method is clearly the best suited for learning anomalies in this setup.

\subsection{A benchmark using the TimeEval library.}

We now look at a broader and more general arrays of problems to evaluate the competitiveness of our method in comparison with state-of-the-art methods.
For that purpose, we use the GutenTAG multivariate datasets, drawn from \cite{WenigEtAl2022TimeEval}. We chose the GutenTAG datasets for the ability to generate a great (1084) number of varied anomaly detection problems; they are mostly formed from inserting anomalies of various lengths in frequency or variance or extremum values into a multivariate time series of 10000 timestamps.
As the anomalies in these dataset seem to have an average size of 100, we set the window sizes of the anomaly detectors to be a fixed $\Delta=100$. Other than that, all other parameters from the previous section are left unchanged.

\begin{figure}[h]
	\centering
	\begin{subfigure}[m]{.7\linewidth}
		\includegraphics[width=\linewidth]{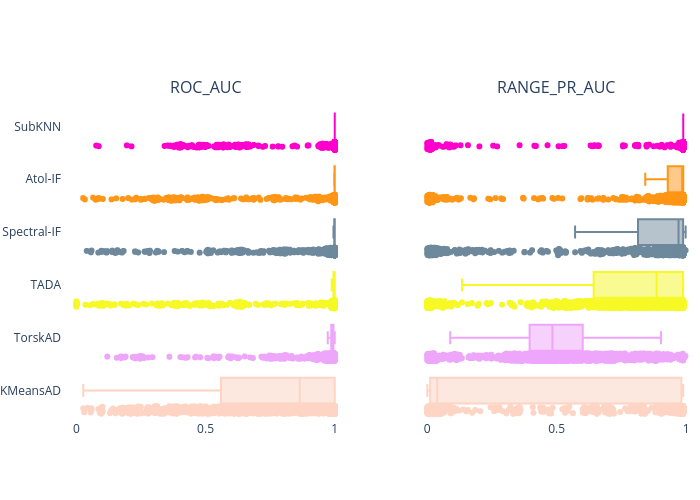}
	\end{subfigure}
	\begin{subfigure}[m]{.28\linewidth}
		\includegraphics[angle=90, width=\linewidth]{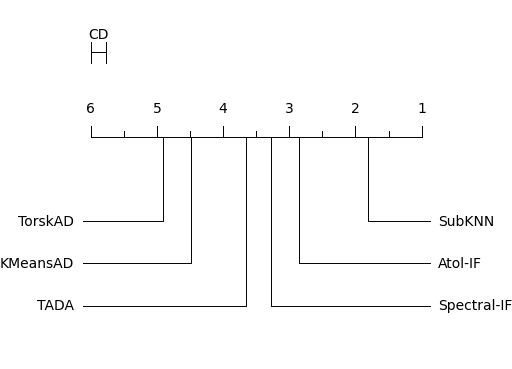}
	\end{subfigure}
	\caption{Left: aggregated results for the TimeEval 1084 GutenTAG synthetic datasets in the form of box plots for the ROC\_AUC and RANGE\_PR\_AUC metrics, and below them the points the metrics have been computed from - where each point represent a metric score from comparing an anomaly score to the underlying truth. Right: autorank summary (see \cite{Herbold2020}) ranking of methods, showing fourth place ranking for our purely topological TADA method.}
	\label{fig:GutenTAG}
\end{figure}

\begin{table}[h]
\begin{center}
	\begin{tabular}{lrrrrr}
\toprule
  algorithm &  \#xp &  \#$>$.9 &  \#rank1 &  med time &  iqr time \\
\midrule
    Atol-IF & 1084 &   844 & 396 &     4.137 &     0.029 \\
   KMeansAD & 1084 &   366 & 213 &     2.327 &     0.375 \\
Spectral-IF & 1084 &   751 & 281 &     1.502 &     0.014 \\
     SubKNN & 1084 &   944 & 939 &     1.244 &     0.502 \\
       TADA & 1084 &   521 & 255 &     4.525 &     0.062 \\
    TorskAD & 1084 &    37 &  32 &     8.459 &     0.250 \\
\bottomrule
\end{tabular}

	\caption{Summary statistics on the GutenTAG problem set for the algorithms evaluated. Even though it is ranked fourth by the statistical pairwise-ranking in Figure \ref{fig:GutenTAG}, our unsupervised method TADA is able to solve roughly half the sets of problems, and being competitive on other 250 of them. The other topological method Atol-IF is ranked second and is able to solve 844 of the 1084 problems. In seconds, the computation median time ("med time") and interquartile range ("iqr time") are standard with respect to the data sizes involved - computations are performed on a single laptop.}
	\label{tab:GutenTAG_stats}
\end{center}
\end{table}

\begin{figure}[h!]
	\centering
	\begin{subfigure}[m]{.52\linewidth}
		\includegraphics[width=\linewidth]{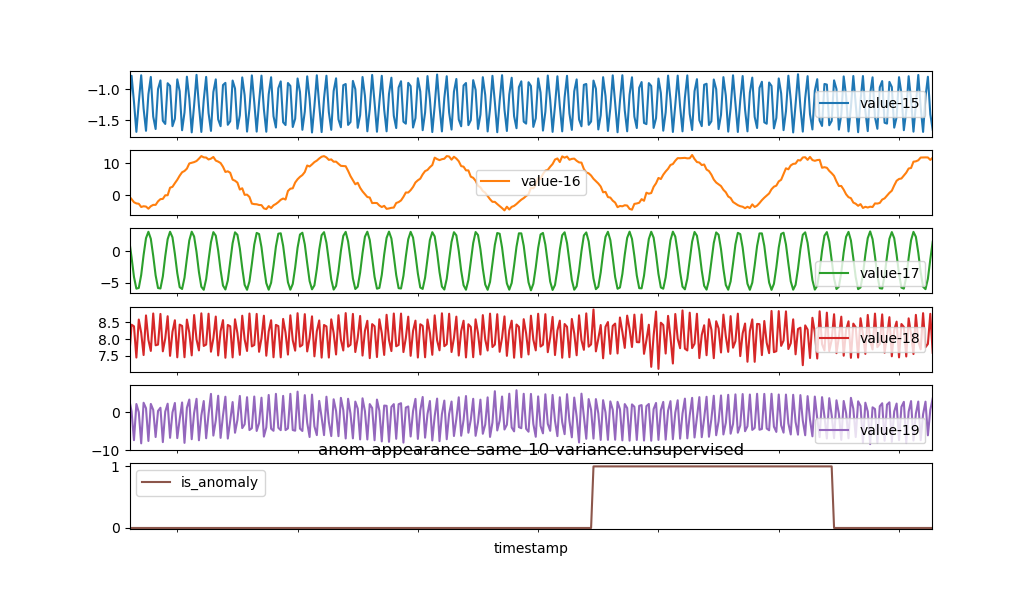}
	\end{subfigure}
	\hspace{-.9cm}
	\begin{subfigure}[m]{.51\linewidth}
		\includegraphics[width=\linewidth]{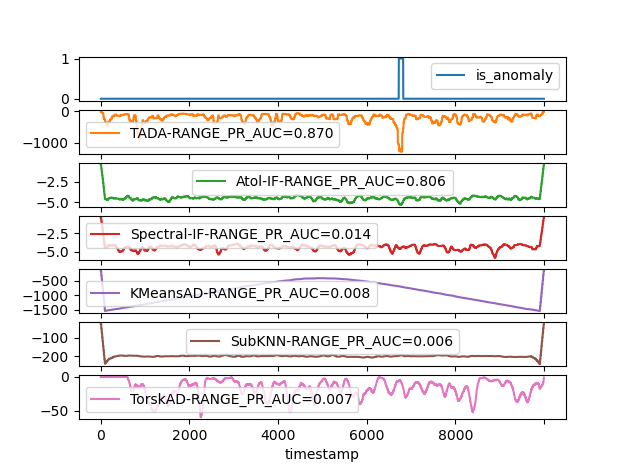}
	\end{subfigure}
	\caption{Left: zoom in on a GutenTAG problem instance (sensors selected for visualization) with a variance anomaly for the last two sensors (purple and red), and the ground truth (last row). Right: problem ground truth (top row) and the corresponding anomaly scores of each method with their RANGE\_PR\_AUC metric in parenthesis. The topological methods TADA and Atol-IF get a good metric score and manage to find the anomalies, while no other method does.}
	\label{fig:GutenTAG_example}
\end{figure}

Statistical summaries and results on the synthetic datasets are shown Figure \ref{fig:GutenTAG}, and in Table \ref{tab:GutenTAG_stats}. As a remainder, the SubKNN, KMeansAD and TorskAD methods were the top three performing methods from the largest anomaly detection benchmark to date (see Table 3 from \cite{SchmidlEtAl2022Anomaly}).
Our TADA procedure manages to solve roughly half the problems and is a top contender among competitors for about a quarter of them. Atol-IF performs better than TADA in this instance, which is not surprising as isolation forest retain much more information from training than TADA, which also implies heavier memory loads. 
Overall SubKNN is able to perform the best on those datasets, and TADA and Atol-IF show good performances, and in some instances only the topological methods manage to solve the problem, see for instance Figure \ref{fig:GutenTAG_example}. These results demonstrate competitiveness of our methodology in the unsupervised anomaly detection learning context.

\subsection{Exathlon real datasets}

Lastly we turn to a real collection of datasets: the 15 Exathlon datasets from \cite{exathlon} consisting in data traces from repeated executions of large-scale stream processing jobs on an Apach Spark cluster, and anomalies are intentional disturbances of those jobs.

\begin{figure}[h]
	\centering
	\begin{subfigure}[m]{.7\linewidth}
		\includegraphics[width=\linewidth]{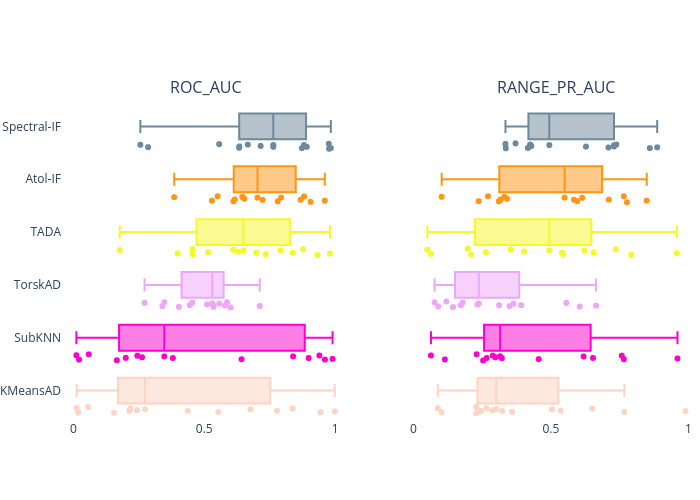}
	\end{subfigure}
	\begin{subfigure}[m]{.28\linewidth}
		\includegraphics[angle=90, width=\linewidth]{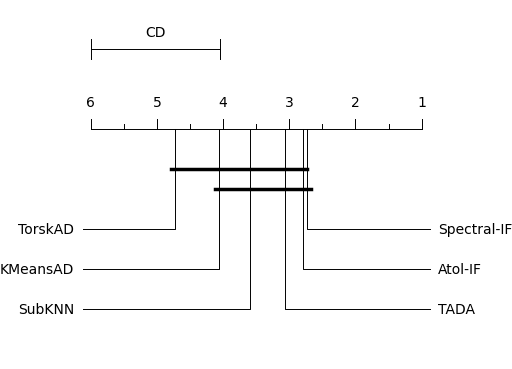}
	\end{subfigure}
	\caption{Left: aggregated results for the 35 Exathlon real datasets in the form of box plots for the ROC\_AUC and RANGE\_PR\_AUC metrics, and below them the points the metrics have been computed from - where each point represent a metric score from comparing an anomaly score to the underlying truth. Right: autorank summary (see \cite{Herbold2020}) ranking of methods on the real datasets, showing second and third place ranking for the topological methods Atol-IF and TADA.}
	\label{fig:Exathlon}
\end{figure}

\begin{table}[h]
	\begin{center}
		\begin{tabular}{lrrrrr}
\toprule
  algorithm &  \#xp &  \#$>$.9 &  \#rank1 &  med time &  iqr time \\
\midrule
    Atol-IF &   15 &     0 &       3 &     9.998 &   448.273 \\
   KMeansAD &   15 &     1 &       1 &     3.336 &     2.568 \\
Spectral-IF &   15 &     0 &       4 &     6.212 &     1.503 \\
     SubKNN &   15 &     1 &       2 &     8.604 &    13.357 \\
       TADA &   15 &     1 &       5 &    11.516 &   450.723 \\
    TorskAD &   15 &     0 &       0 &    34.282 &    37.395 \\
\bottomrule
\end{tabular}

		\caption{Summary statistics on the Exathlon real data problems for the algorithms evaluated. Even though it is ranked third by statistical ranking in Figure \ref{fig:Exathlon}, our unsupervised method TADA is the top RANGE\_PR\_AUC score ($\#\text{rank}1$ column) over all problems, which hints that it is able to solve different sorts of anomaly detection problems than the others. In seconds, the computation median time ("med time") and interquartile range ("iqr time") are high for the topological methods, see commentaries in the text. Computations are performed on a single laptop.}
		\label{tab:Exathlon_stats}
	\end{center}
\end{table}

Using the same metrics, collection of anomaly detection methods and exact same calibration as in the previous TimeEval experiment, we produce the following results. The main statistical summaries are presented on Figure \ref{fig:Exathlon} and Table \ref{tab:Exathlon_stats}.
Overall the topological methods are strong competitors for these datasets, with TADA coming off as the most often number one ranked method.
Due to the real nature of the datasets it is not surprising that the studied methods do not "solve" them in a way those methods were able to solve the GutenTAG datasets of the TopologicalWheels datasets. We show in Figure \ref{fig:Exathlon_example} the one instance where TADA is able to solve the problem completely, and highlight that is has happened without any calibration.

\begin{table}[h]
	\begin{center}
		\begin{tabular}{lrrr}
\toprule
        dataset &  TADA computing time (s) &  n\_sensors &  n\_timestamps \\
\midrule
10\_2\_1000000\_67 &        1127.328406 &        165 &         10250 \\
10\_3\_1000000\_75 &          10.981673 &          8 &         46656 \\
10\_4\_1000000\_79 &           7.367096 &          3 &         43086 \\
  2\_2\_200000\_69 &         121.436984 &        130 &          2874 \\
 3\_2\_1000000\_71 &         457.991445 &        194 &          2474 \\
  3\_2\_500000\_70 &         591.719635 &        208 &          2611 \\
  5\_1\_500000\_62 &          14.912100 &         16 &         46660 \\
 5\_2\_1000000\_72 &         460.615526 &        195 &          2481 \\
  6\_1\_500000\_65 &          11.515798 &         12 &         46649 \\
  6\_3\_200000\_76 &           9.138171 &          7 &         46654 \\
  8\_3\_200000\_73 &           8.023327 &          4 &         46641 \\
 8\_4\_1000000\_77 &           7.096984 &          2 &         43078 \\
 9\_2\_1000000\_66 &        3046.870513 &        239 &          7481 \\
  9\_3\_500000\_74 &           9.379143 &          7 &         46650 \\
 9\_4\_1000000\_78 &           7.405223 &          3 &         43105 \\
\bottomrule
\end{tabular}

		\caption{Summary computation times for TADA and problem sizes on the Exathlon real datasets.}
		\label{tab:Exathlon_tada_stats}
	\end{center}
\end{table}

\begin{figure}[h!]
	\centering
	\begin{subfigure}[m]{.49\linewidth}
		\includegraphics[width=\linewidth]{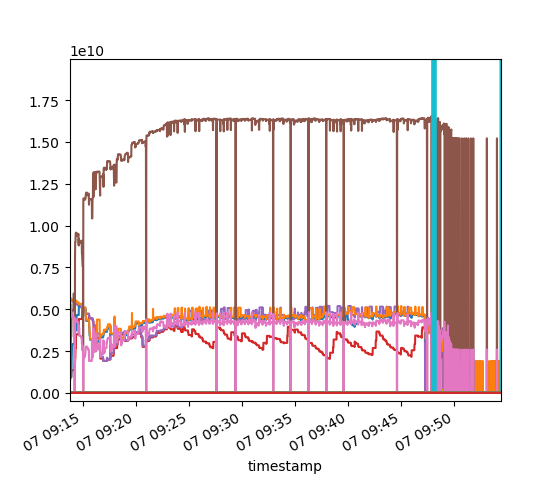}
	\end{subfigure}
	\hspace{-.9cm}
	\begin{subfigure}[m]{.55\linewidth}
		\includegraphics[width=\linewidth]{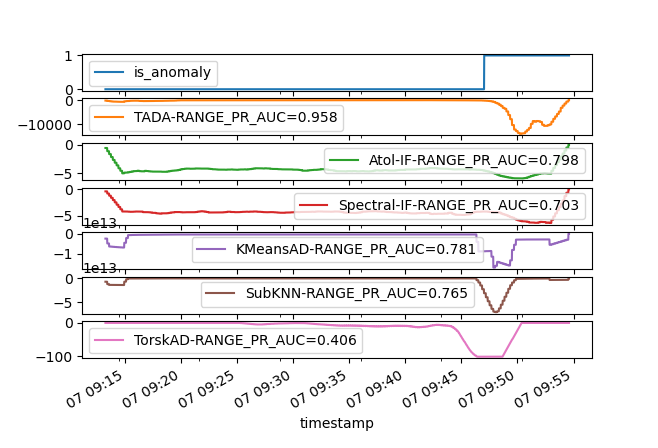}
	\end{subfigure}
	\caption{Left: zoom in on dataset 3\_2\_1000000\_71. Right: ground truth (top row) and anomaly scores for the six methods and their RANGE\_PR\_AUC score in parenthesis. While all methods capture locate the beginning of the anomaly period correctly, only TADA manages to catch it in its entirety.}
	\label{fig:Exathlon_example}
\end{figure}

One drawback of the topological methods appearing here is the high variance in execution time, which originates from computing topological features on a great number of sensors, see Table \ref{tab:Exathlon_tada_stats} for a scaling intuition.
As our implementation of Algorithm \ref{algo:detection_score} is naive, we point that there are strategies for optimizing computation times: ripser, subsampling, clustering that makes sense, etc. Those strategies are outside the scope of this paper.

\subsection{Score interpretability}

The anomaly score we introduce is constructed from estimating the mean measure of persistence diagrams supported by $K$ centroids, and analysing the resulting embedding's main distribution features. Once these centers are learnt it is possible to engineer anomaly scores respective to a particular center, or possibly to a set of centers e.g. centers associated with a particular homology dimension. Let us examine this first possibility, and introduce the center-targeted scores:
\begin{align}
	\tilde{s}_i = \hat{\Sigma}_{ii}^{-1/2} |v_i - \hat{\mu}_i|,
\end{align}
where $\hat{\mu}, \hat{\Sigma}$ are the estimated mean and covariance of the vectorization $v$ of Algorithm \ref{algo:vec} (time indices are implied and ommited for this discussion). These scores can be interpreted as testing for anomalies with respect to a single embedding dimension as if the vectorization had independent components. These center-targeted scores allow to analyze an original anomaly by looking at the score deviations of each vector component. Because the vector components are integrated from a learnt centroid, the scores can be traced back to a specific region in $\mathbb{R}^2$, see for instance Figure \ref{fig:tw-interpretability}.

\begin{figure}[h]
	\centering
	\begin{subfigure}[m]{.55\linewidth}
		\centering
		\includegraphics[width=\linewidth]{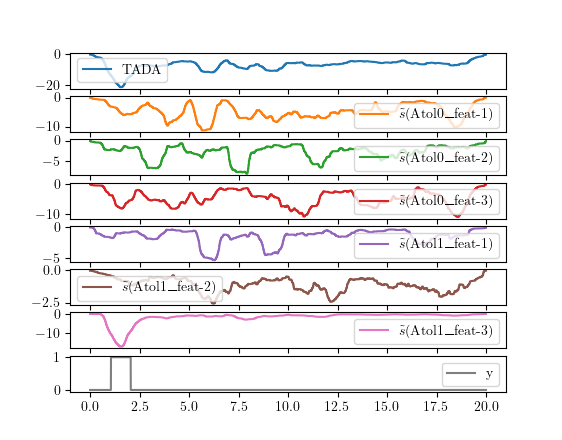}
	\end{subfigure}
	\hspace{-1cm}
	\begin{subfigure}[m]{.5\linewidth}
		\centering
		\includegraphics[width=\linewidth]{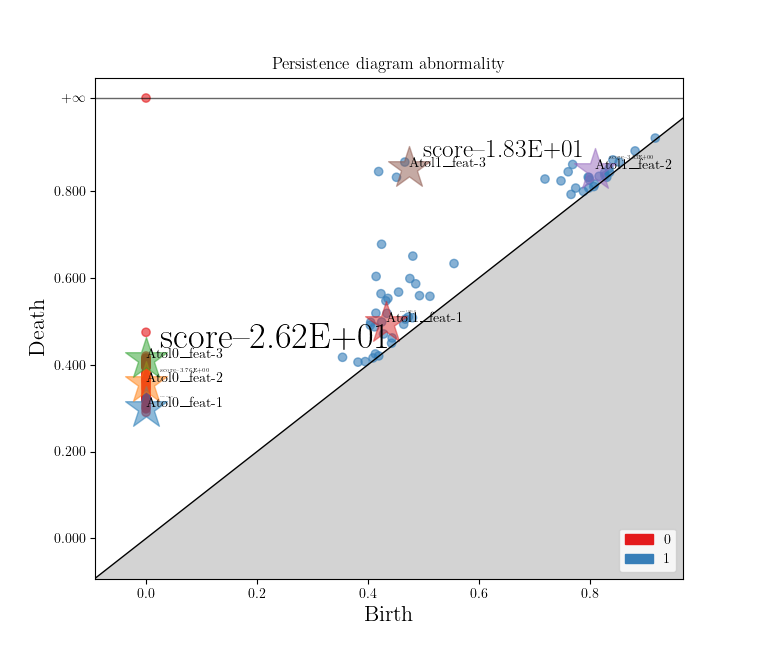}
	\end{subfigure}
	\caption{Left: Overall score (top curve, in blue), scores for $\tilde{s}_{i}$ (middle curves, three top ones for features corresponding to homology dimension 0, three bot ones for features corresponding to homology dimension 1) and ground truth (bottom curve) on a single Topological Wheels dataset. The last homology dimension 1 has the strongest correspondance to the overall score, and matches the underlying truth almost exactly. Right: $\tilde{s}_{i}$ scores of an abnormal persistence diagram on this dataset, next to the associated quantization centers (colored stars) of Algorithm \ref{algo:kmeans_batch}. The scores are written in a font size proportional to them, so that the more abnormal scores appear bigger. In this instance the quantization centers in dimension 0 and 1 with the highest persistence react to this diagram, hinting at a change in the latent data structure as the highest persistence diagram points are usually associated with signal in comparison with points nearer the diagonal.}
	\label{fig:tw-interpretability}
\end{figure}

This leads to valuable interpretation. For instance if an abnormal score of TADA were to be caused by a large deviation in a homology dimension 1 center-related component, it is likely that at that time an abnormal dependence cycle is created for a longer or shorter period of time than for the rest of the time series, therefore that the dependence pattern has globally changed in that period of time. See for instance an illustration on the Topological Wheels problem in Figure \ref{fig:tw-interpretability} where globally changing the dependence pattern between sensors is exactly how the abnormal data was produced.
And in the case that the produced score is abnormal simply by virtue of a shift in several dimension 0 centers-related components, it signifies an anomaly in the connectivity dependence structure that does not affect the higher order homological features, therefore it could be attributed to a default (such as a breakdown) in one of the original sensors for instance.

\section{Conclusion}

It is common knowledge that no anomaly detection method can help with identifying all kinds of anomalies. The framework introduced in this paper is relevant for detecting abnormal topological changes in the dependence structure of data, and turns out to be competitive with respect to other state-of-the-art approaches on various benchmark data sets. Naturally, there are many different sorts of anomalies that the proposed method is not able to detect at all - for instance, as the topological embedding is invariant to graph isomorphism, any anomalies linked to node permutation (change of node labelling) cannot be caught. The same is true for homothetic \textit{jumps}: when signals would simultaneously get identically multiplied, the correlation-based similarity matrix would remain unchanged, leading to unchanged topological embedding. While such invariances can be thought of as hindrances, they can also come as a welcome feature if those anomalies are in fact built-in the considered applied problem - for instance in the case of labeling uncertainties in sensors.

The topological anomaly detection finds anomalies that other methods do not seem to discover. It is generally understood that topological information is a form of global information that is complementary to the information gathered by more traditional approaches, e.g. spectral detectors. While confirming this, the above numerical experiment also suggest that topological information is commonly present in various real or synthetic datasets. Therefore for practical applied purposes it is probably best to use our method in combination with other dedicated methods, for instance one that focuses on "local" data shifts such as the \textsc{SubKNN} method.

Focusing on the case of detection of anomalies in the dependence structure of multivariate time series, it appears that the only parameter that requires a careful tuning in our method is the window size (or temporal resolution) $\Delta$, as for most of existing procedures (see Section \ref{sec:applications}). Designing methods to empirically tune this window size, or to combine the outputs of our method at different resolutions would be a relevant addition to our work, that is left for future research.

Let us now emphasize the broader flexibility of the framework we introduce. First, it is not tied to detect changes in correlation structures: we may use Algorithm \ref{algo:timeserie_to_pers_diagrams} with other dissimilarity measures between channels, and in fact we may build persistence diagrams from a time series of more general metric spaces - e.g. meshed shapes, images... -   
in the more general case (as done in \cite{Chazal21} for graphs).  Second, the vectorization we propose with Algorithm \ref{algo:kmeans_batch} and \ref{algo:vec} does not necessarily take a sequence of persistence diagrams as input: any sequence of measures may be vectorized the same way. It may find applications in monitoring sequences of point processes realizations, as in evolution of distributions of species for instance - see, e.g., \cite{Renner15}. And finally, one may process the output of the vectorization procedure in other ways than building an anomaly score. For instance, using these vectorizations as inputs of any neural network, or change-points detection procedures such as KCP (\cite{Arlot19}) could provide a dedicated method to retrieve change points of a global structure.

\section{Acknowledgements}

This work has been supported by the French government under the "France 2030” program, as part of the SystemX Technological Research Institute within the Confiance.ai project. The authors also thank the ANR TopAI chair in Artificial Intelligence (ANR–19–CHIA–0001) for financial support.

The authors are grateful to Bernard Delyon for valuable comments and suggestions concerning convergence rates in $\beta$-mixing cases.

\section{Proofs}\label{sec:proofs}

Most of the proposed results are adaptation of proofs in the independent case to the dependent one. A peculiar interest lies in concentration results in this framework, we list important ones in the following section.

\subsection{Probabilistic tools for $\beta$-mixing concentration}
In the derivations to follow extensive use will be made of a consequence of Berbee's Lemma. 
\begin{lem}{\cite[Proposition 2]{Doukhan95}}\label{lem:betamix_indep}
Let $(X_i)_{i \geqslant 1}$ be a sequence of random variables taking their vales in a Polish space $\mathcal{X}$, and, for $j>0$, denote by 
\[
b_j = \E \left [ \sup_{B \in \sigma(j+1, + \infty)} | \P(B \mid \sigma(- \infty,j)) - \P(B)| \right ]. \] 
Then there exists a sequence $(\tilde{X}_i)_{i \geqslant 1}$ of independent random variables such that, for any $i \geqslant 1$, $\tilde{X}_i$ and $X_i$ have the same distribution and $\P(X_i \neq \tilde{X}_i) \leqslant b_i$.
\end{lem} 
The above Lemma allows to translate standard concentration bounds from the i.i.d. framework to the dependent case, where dependency is seized in terms of $\beta$-mixing coefficients.
 
  Let us recall here the general definition of $\beta$-mixing coefficients from Definition \ref{def:beta_mixing}. For a sequence of random variables $(Z_t)_{t \in \Z}$ (not assumed stationary), the \textit{beta-mixing} coefficient of order $q$ is 
\begin{align*}
\beta(q)=\sup_{t \in \Z} \E \left [ \sup_{B \in \sigma(t+q,+ \infty)} | \P(B\mid \sigma(- \infty, t)) - \P(B) | \right ].
\end{align*}
If the sequence $(Z_t)_{t \in \Z}$ is assumed to be stationary, $\beta(q)$ may be written as
\begin{align*}
\beta(q) = \E (\d_{TV}(P_{(X_q, X_{q+1}, \hdots)  \mid \sigma(\hdots, X_0)},P_{(X_q, X_{q+1}, \hdots)})),
\end{align*}
where $\d_{TV}$ denotes the total variation distance. We will make use of the following adaptation of Bernstein's inequality to the dependent case.
\begin{theorem}\label{thm:concentration_Doukhan}{\cite[Theorem 4]{Doukhan94}}
Let $(X_t)_{t \in \Z}$ be a sequence of (real) variables with $\beta$-mixing coefficients $(\beta(q))_{q \in \N^*}$, that satisfies
\begin{enumerate}
\item $\forall t \in \Z \quad \E(X_t) = 0$,
\item $\forall t,n \in \Z \times \N \quad \E \left | \sum_{j=1}^n X_{t+j} \right |^2 \leqslant n \sigma^2$,
\item $\forall t \quad |X_t| \leqslant M$ a.s..
\end{enumerate}
Then, for every $x \geqslant 0$, 
\begin{align*}
\P \left ( \frac{1}{n} \left | \sum_{t=1}^n X_t \right | \geqslant 2 \sigma \sqrt{\frac{x}{n}} + \frac{4Mx}{3n} \right ) \leqslant 4 e^{-x} + 2 \beta \left ( \lceil \frac{n}{17} \rceil - 1 \right ).
\end{align*}
\end{theorem}
To apply Theorem \ref{thm:concentration_Doukhan}, a bound on the variance term is needed. Such bounds are available in the stationary case under slightly milder assumptions (see, e.g., \cite{Rio93}). For our purpose, a straightforward application of \cite[Theorem 1.2, a)]{Rio93} will be sufficient, exposed below.
\begin{lem}\label{lem:cov_betamelange}
Let  $X_t$ denote a centered and stationary sequence of real variables with $\beta$-mixing coefficients $(\beta(q))_{q \in \N^*}$, such that $|X_t| \leqslant M$ a.s.. 

Then it holds
\begin{align*}
\frac{1}{n} \E \left ( \sum_{j=1}^n X_j \right )^2 \leqslant 4M^2 \int_0^1 \beta^{-1}(u)du,
\end{align*}
where $\beta^{-1}(u) = \sum_{ k \in \N} \1_{\beta(k) > u}$.
\end{lem}

\subsection{Proofs for Section \ref{sec:theoretical_results}}

\subsubsection{Proof of Theorem \ref{thm:cv_batch}}
\begin{proof}[Proof of Theorem \ref{thm:cv_batch}]

We begin by the proof of Theorem \ref{thm:cv_batch}. It follows the proof of \cite[Theorem 9]{Chazal21} in the i.i.d. case, with adaptations to cope with dependency using Lemma \ref{lem:betamix_indep}. 

To apply Lemma \ref{lem:betamix_indep}, first note that the space $\mathcal{M}(R,M)$, endowed with the Levy-Prokhorov metric, is a Polish space (see, e.g., \cite[Theorem 1.11]{Prokhorov56}). Using Lemma \ref{lem:betamix_indep} as in \cite[Proof of Proposition 2]{Doukhan95} yields the existence of $\tilde{X}_1, \hdots, \tilde{X}_n$ such that, denoting by $Y_k$ (resp. $\tilde{Y}_k$) the vector $(X_{(k-1)q+1}, \hdots, X_{kq})$ (resp. $(\tilde{X}_{(k-1)q+1}, \hdots, \tilde{X}_{kq})$), for $1 \leqslant k \leqslant n/q$, it holds:
\begin{itemize}
\item For every $k \geqslant 1$ $Y_k$ has the same distribution as $\tilde{Y}_k$, and $\P(\tilde{Y}_k \neq Y_k) \leqslant \beta(q)$.
\item The random variables $(Y_{2k})_{k \geqslant 1}$ are independent, as well as the variables $(Y_{2k-1})_{k \geqslant 1}$.
\end{itemize} 

For any $\cb \in \B(0,R)^k$, we denote by $\hat{m}(\cb)$ (resp. $\tilde{m}(\cb)$) the vector of centroids defined by, for all $j=1, \hdots, k$,
\begin{align*}
\hat{m}(\cb)_j = \frac{\bar{X}_n(du) \left [u \1_{W_j(\cb)}(u) \right ]}{\hat{p}_j(\cb)}, \qquad \tilde{m}(\cb)_j = \frac{\bar{\tilde{X}}_n(du) \left [u \1_{W_j(\cb)}(u) \right ]}{\tilde{p}_j(\cb)}, 
\end{align*}
where $\hat{p}_j(\cb)$ (resp. $\tilde{p}_j(\cb)$) denotes $\bar{X}_n(W_j(\cb))$ (resp. $\bar{\tilde{X}}_n(W_j(\cb))$), adopting the convention $\hat{m}_j(\cb), \tilde{m}_j(\cb)=0$ when the corresponding cell weight is null.

The following lemma ensures that $\hat{m}(\cb)$ contracts toward $\cb^*$, provided $\cb \in \B(\cb^*,R_0)$.

\begin{lem}\label{lem:batch_onestep}
With probability larger than $1-16 e^{-c_1 n p_{\min}/(qM)} -  2e^{-x}$, it holds, for every $\cb \in \B(\cb^*,R_0)$, 
\begin{align*}
\|\hat{m}(\cb) - \cb^*\|^2 \leqslant \frac{3}{4} \| \cb-\cb^*\|^2 + C_1 \frac{D_{n/q}^2}{p_{\min}^2} + C_2 \frac{k R^2 M^2}{p_{\min}^2} \left ( \frac{1}{n}\sum_{i=1}^n \1_{X_i \neq \tilde{X}_i} \right )^2, 
\end{align*}
where $D_{n/q} = \frac{RM\sqrt{q}}{\sqrt{n}} \left ( k \sqrt{d \log(k)} + \sqrt{x} \right )$.
\end{lem} 
The proof of Lemma \ref{lem:batch_onestep} is postponed to Section \ref{sec:proof_lemma_batch_onestep}. Equipped with Lemma \ref{lem:batch_onestep}, we first prove recursively that, if $\cb^{(0)} \in \B(\cb^*,R_0)$, then w.h.p., for all $t >0$ $\cb^{(t)} \in \B(\cb^*,R_0)$. We let $\Omega_1$  be defined as
\begin{align*}
\Omega_1 & = \left \{ C_2 {k R^2 M^2} \left ( \frac{1}{n}\sum_{i=1}^n \1_{X_i \neq \tilde{X}_i} \right )^2/p_{\min}^2 \leqslant R_0^2/8 \right \}.
\end{align*}
Noting that $\E \left ( \frac{1}{n}\sum_{i=1}^n \1_{X_i \neq \tilde{X}_i} \right )^2 \leqslant \E \left ( \frac{1}{n}\sum_{i=1}^n \1_{X_i \neq \tilde{X}_i} \right ) = \beta(q)$, Markov inequality yields 
\begin{align*}
\P(\Omega_1^c) \leqslant C \frac{k M^2}{\kappa_0^2 p_{\min}^2} \beta(q).
\end{align*}
Choosing $x = c_1 (n/q) \kappa_0^2 p_{\min}^2/M^2$ in Lemma \ref{lem:batch_onestep}, for $c_1$ small enough yields, for $(n/q)$ large enough, 
\begin{align*}
\|\hat{m}(\cb) - \cb^*\|^2 \leqslant \frac{3}{4} R_0^2 + \frac{R_0^2}{8} + \frac{R_0^2}{8} = R_0^2,
\end{align*} 
with probability larger than $1- 18e^{-c_1 n \kappa_0^2 p_{\min}^2/qM^2} - C \frac{k M^2}{\kappa_0^2 p_{\min}^2} \beta(q)$, provided $\cb \in \B(\cb^*,R_0)$. Denoting by $\Omega_2$ the probability event onto which the above equation holds, a straightforward recursion entails that, if $\cb^{(0)} \in \B(\cb^*,R_0)$, then, for all $t \geqslant 1$ $\cb^{(t)} = \hat{m}(\cb^{(t-1)}) \in \B(\cb^*,R_0)$, on $\Omega_2$.

Then, using Lemma \ref{lem:batch_onestep} iteratively yields that, on $\Omega_2 \cap \Omega_x$, where $\P(\Omega_x^c) \leqslant 2e^{-x}$, for all $t \geqslant 1$, provided $\cb^{(0)} \in \B(\cb^*,R_0)$, 
\begin{align}\label{eq:batch:tstep}
\|\cb^{(t)} - \cb^*\|^2 \leqslant \left ( \frac{3}{4} \right )^t\|\cb^{(0)} - \cb^*\|^2 + C_1 \frac{D_{n/q}^2}{p_{\min}^2} + C_2 \frac{k R^2 M^2}{p_{\min}^2} \left ( \frac{1}{n}\sum_{i=1}^n \1_{X_i \neq \tilde{X}_i} \right )^2. 
\end{align}

Theorem \ref{thm:cv_batch} now easily follows. For the first inequality, let $t \geqslant 1$, then, using Markov inequality again gives
\begin{align*}
\P \left ( \frac{1}{n} \sum_{i=1}^n \1_{X_i \neq \tilde{X}_i} \geqslant \sqrt{q/n} \right ) \leqslant \sqrt{\frac{n}{q}} \beta(q).
\end{align*}
Thus, the assumption $\beta^2(q)/q^3 \leqslant n^{-3}$ entails that
\begin{align*}
\|\cb^{(t)} - \cb^*\|^2 \leqslant \left ( \frac{3}{4} \right )^t R_0^2 + C_1 \frac{D^2_{n/q}}{p_{\min}^2} + C_2 \frac{kR^2M^2}{p_{\min}^2(n/q)},
\end{align*}
with probability larger than $1- 18e^{-c_1 n \kappa_0^2 p_{\min}^2/qM^2} - C \frac{k M^2}{\kappa_0^2 p_{\min}^2} \beta(q) - q/n-2e^{-x}$ that is larger than $1 - C \frac{q k M^2}{n\kappa_0^2 p_{\min}^2} - 2e^{-x}$.

For the second inequality, denote by $Z_t$ the random variable
\begin{align*}
& Z_t =  \\
& \quad \left ( \|\cb^{(t)} - \cb^*\|^2 - \left ( \frac{3}{4} \right )^t\|\cb^{(0)} - \cb^*\|^2 - C_1 \frac{q R^2M^2 k^2d\log(k)}{np_{\min}^2} - C_2 \frac{k R^2 M^2}{p_{\min}^2} \left ( \frac{1}{n}\sum_{i=1}^n \1_{X_i \neq \tilde{X}_i} \right )^2\right )_+ \1_{\Omega_2},
\end{align*}
and remark that \eqref{eq:batch:tstep} entails
\begin{align*}
\P \left (Z_t \geqslant C \frac{R^2M^2q}{n} x \right ) \leqslant \P(\Omega_x^c) \leqslant 2e^{-x}.
\end{align*}
We deduce that
\begin{align*}
\E(Z_t) \leqslant  C \frac{qR^2M^2}{n},
\end{align*}
which leads to
\begin{align*}
\E(\|\cb^{(t)} - \cb^*\|^2) & \leqslant \E ( \|\cb^{(t)} - \cb^*\|^2 \1_{\Omega_2}) + 4k R^2 M \P(\Omega_2^c) \\
& \leqslant \E(Z_t) +  \left ( \frac{3}{4} \right )^t\|\cb^{(0)} - \cb^*\|^2 + C_1 \frac{q R^2M^2 k^2d\log(k)}{np_{\min}^2} \\
& \quad + C_2 \frac{kR^2M^2}{p_{\min}^2} \E \left ( \left ( \frac{1}{n}\sum_{i=1}^n \1_{X_i \neq \tilde{X}_i} \right )^2\right ) + 4k R^2 M \P(\Omega_2^c).
\end{align*}
Noting that 
\begin{align*}
\E \left ( \left ( \frac{1}{n}\sum_{i=1}^n \1_{X_i \neq \tilde{X}_i} \right )^2\right ) \leqslant \E \left (  \frac{1}{n}\sum_{i=1}^n \1_{X_i \neq \tilde{X}_i} \right ) = \beta(q),
\end{align*}
and using 
\begin{align*}
\P(\Omega_2^c) & \leqslant \left ( e^{-c_1 n \kappa_0^2 p_{\min}^2/qM^2} + \frac{k M^2}{\kappa_0^2 p_{\min}^2} \beta(q)  \right ) \\
& \leqslant C\frac{kM^2}{\kappa_0^2 p_{\min}^2} (\beta(q) + (q/n)) \\
& \leqslant C\frac{q kM^2}{n\kappa_0^2 p_{\min}^2}
\end{align*}
whenever $\beta(q) \leqslant q/n$ leads to the result.
\end{proof}

\subsubsection{Proof of Theorem \ref{thm:cv_minibatch}}

\begin{proof}[Proof of Theorem \ref{thm:cv_minibatch}]
This proof follows the steps of \cite[Proof of Lemma 18]{Chazal21}.

As in the proof of Lemma \ref{lem:batch_onestep}, let $\tilde{X}_1, \hdots, \tilde{X}_n$ be such that, denoting by $Y_k$ (resp. $\tilde{Y}_k$) the vector $(X_{(k-1)q+1}, \hdots, X_{kq})$ (resp. $(\tilde{X}_{(k-1)q+1}, \hdots, \tilde{X}_{kq})$), for $1 \leqslant k \leqslant n/q$, it holds:
\begin{itemize}
\item For every $k \geqslant 1$ $Y_k$ has the same distribution as $\tilde{Y}_k$, and $\P(\tilde{Y}_k \neq Y_k) \leqslant \beta(q)$.
\item The random variables $(Y_{2k})_{k \geqslant 1}$ are independent, as well as the variables $(Y_{2k-1})_{k \geqslant 1}$.
\end{itemize}
Let $A_{\indep}$ denote the event
\begin{align*}
A_\indep = \left \{ \forall j =1, \hdots, n/q \quad Y_j = \tilde{Y}_j  \right \}.
\end{align*}
A standard union bound yields that $\P(A_\indep^c) \leqslant \frac{n}{q} \beta(q)$. On $A_\indep$, the minibatches used by Algorithm \ref{algo:kmeans_minibatch} may be considered as independent, so that the main lines of  \cite[Proof of Lemma 18]{Chazal21} readily applies, replacing $X_i$'s by $\tilde{X}_i$'s. In what follows we let $\tilde{\cb}^{(t)}$ denote the output of the $t$-th iteration of Algorithm \ref{algo:kmeans_minibatch} based on $\tilde{X}_1, \hdots, \tilde{X}_n$.

Assume that $n \geqslant k$, and $q \geqslant C \frac{M^2}{p_{\min}^2} \log(n)$, for a large enough constant $C$ that only depends on $\int_0^1 \beta^{-1}(u)du$, to be fixed later.   
For $t \leqslant n/(4q)=T$, let $A_{t,1}$ and $A_{t,3}$ denote the events
\begin{align*}
A_{t,1} & = \left \{ \forall j = 1, \hdots, k \quad |\hat{p}_j(t) - p_j(t)| \leqslant \frac{p_{\min}}{128} \right \}, \\
A_{t,3} & = \left \{ \forall j = 1, \hdots, k \quad \left \| \int (\tilde{c}_j^{(t)}-u) \1_{W_j(\tilde{\cb}^{(t)})}(u)(\bar{\tilde{X}}_{B_t^{(3)}} - \E(X))(du) \right \| \leqslant 8 R \sqrt{\frac{M k d p_{\min}}{C}} \right \},
\end{align*} 
where $\hat{p}_j(t) = \bar{\tilde{X}}_{B_t^{(1)}}(W_j(\tilde{\cb}^{(t)}))$. Then, according to Theorem \ref{thm:concentration_Doukhan} with $x = 2 \log(n)$ and Lemma \ref{lem:cov_betamelange} to bound the corresponding $\sigma$, for $j \in \{1,3\}$, $\P(A_{t,j}) \leqslant 4 d k/n^2 + 2 k d \beta(q_n/18)$, for $n$ large enough. 

Further, define
\begin{align*}
A_{\leqslant t} = \bigcap_{j \leqslant t} A_{j,1} \cap A_{j,3}.
\end{align*}
Then, provided that $q \geqslant c_0 \frac{k^2d M^2}{p_{\min}^2 \kappa_0^2}\log(n)$, where $c_0$ only depends on $\int_0^1 \beta^{-1}(u) du$, we may prove recusively that
\begin{align*}
\forall p \leqslant t \quad \tilde{\cb}^{(p)} \in \B(\cb^*,R_0)
\end{align*} 
on $A_{\leqslant t}$ whenever $\tilde{\cb}^{(0)} = \cb^{(0)} \in \B(\cb^*,R_0)$ (first step of the proof of \cite[Lemma 18]{Chazal21}).

Next, denoting by $a_t = \| \tilde{\cb}^{(t)} - \cb^*\|^2\1_{A_{\leqslant t}}$, we may write
\begin{align*}
\E(a_{t+1}) \leqslant \E(\| \tilde{\cb}^{(t+1)} - \cb^*\|^2 \1_{A_{t+1,1}} \1_{A_{\leqslant t}}) + R_1,
\end{align*}
with 
\begin{align*}
R_1 & \leqslant 4kR^2 \left ( \P(A_{t+1,3}^c) \right ) \\
& \leqslant 16 k^2 d  R^2 \left ( n^{-2} + \beta(q/18) \right ) \\
& \leqslant 32 k^2 d R^2 (q/n)^2,
\end{align*} 
recalling that $\beta(q/18)/q^2 \leqslant n^{-2}$. 
Proceeding as in \cite[Proof of Lemma 18]{Chazal21}, we may further bound
\begin{align*}
\E(\| \tilde{\cb}^{(t+1)} - \cb^*\|^2 \1_{A_{t+1,1}} \1_{A_{\leqslant t}}) \leqslant \left ( 1 - \frac{2-K_1}{t+1} \right ) \E(a_t) + \frac{12k d M R^2}{p_{\min}(t+1)^2},
\end{align*}
for some $K_1 \leqslant 0.5$. Noticing that $k \leqslant M/p_{\min}$ and $t+1 \leqslant T = n/(4q)$ yields that
\begin{align*}
\E(a_{t+1}) \leqslant \left ( 1 - \frac{2-K_1}{t+1} \right ) \E(a_t) + \frac{14kdMR^2}{p_{\min}(t+1)^2}.
\end{align*}
Following \cite[Proof of Theorem 10]{Chazal21}, a standard recursion entails
\begin{align*}
\E(a_t) \leqslant \frac{28kdMR^2}{p_{\min}t},
\end{align*}
for $t \leqslant n/(4q)$. At last, since $ \| \tilde{\cb}^{(T)} - \cb^*\|^2 \1_{A_\indep} = \| \cb^{(T)} - \cb^*\|^2 \1_{A_\indep}$, we conclude that
\begin{align*}
\E \| \cb^{(T)} - \cb^*\|^2 & \leqslant  \E(\| \tilde{\cb}^{(T)} - \cb^* \|^2) + 4 k R^2 \P(A_\indep^c) \\
& \leqslant \E(\| \tilde{\cb}^{(T)} - \cb^* \|^2 \1_{A_{\leqslant T}}) + 4 k R^2 \P(A_{\leqslant T}^c) + \frac{4kR^2}{(n/q)} \\
& \leqslant \E(a_T) + \frac{16 k^2 R^2 d}{(n/q)} \\
& \leqslant 128 \frac{ k M  R^2 d}{p_{\min}(n/q)},
\end{align*}
where $k \leqslant M/p_{\min}$ and $T = \frac{(n/q)}{4}$ have been used.
\end{proof}
\subsubsection{Proof of Lemma \ref{lem:batch_onestep}}\label{sec:proof_lemma_batch_onestep}
\begin{proof}[Proof of Lemma \ref{lem:batch_onestep}]
Assume that $\cb \in \B(\cb^*,R_0)$, for some optimal $\cb^* \in \mathcal{C}_{opt}$. Then, for any $K>0$, it holds
\begin{align}\label{eq:onestep_connection_unmixing}
\|\hat{m}(\cb) - \cb^*\|^2 & \leqslant (1+K) \|\tilde{m}(\cb) - \cb^*\|^2 + (1+K^{-1}) \|\hat{m}(\cb) - \tilde{m}(\cb)\|^2.
\end{align}

The first term of the right hand side may be controlled using a slight adaptation of \cite[Lemma 22]{Chazal21}.
\begin{lem}\label{lem:concentration_batch}
With probability larger than $1 - 16e^{-x}$, for all $\cb \in \B(0,R)^k$ and $j \in [\![1,k]\!]$, it holds
\begin{align*}
\tilde{p}_j(\cb) &\geqslant p_j(\cb) - \sqrt{p_j(\cb)} \sqrt{\frac{8 M c_0 q \log(k) \log(2 n N_{\max})}{n} + \frac{8 M q x}{n}}, \\
\tilde{p}_j(\cb) & \leqslant p_j(\cb)  + \frac{8 M c_0 q \log(k) \log(2 n N_{\max})}{n} + \frac{8 M q x}{n} \\
& \qquad  + \sqrt{\frac{8 M c_0 q \log(k) \log(2 n N_{\max})}{n} + \frac{8 M q x}{n}} \sqrt{p_j(\cb)},
\end{align*} 
where $c_0$ is an absolute constant. Moreover, with probability larger than $1 - 2e^{-x}$, it holds
\begin{multline*}
\sup_{\cb \in \B(0,R)^k} \left \| \left ( \int (c_j - u)\1_{W_j(\cb)}(u) (\bar{\tilde{X}}_n - \E(X))(du)  \right )_{j=1, \hdots, k} \right \| \\ \leqslant C_0 \frac{RM \sqrt{q}}{\sqrt{n}} \left ( k \sqrt{d \log(k)} + \sqrt{x} \right ),
\end{multline*}
where $C_0$ is an absolute constant.
\end{lem}
\begin{proof}[Proof of Lemma \ref{lem:concentration_batch}]
We intend here to recover the standard i.i.d.  bounds given in \cite[Lemma 22]{Chazal21}. To this aim, we let $\tilde{p}_{j,0}(\cb)$ and $\tilde{p}_{j,1}(\cb)$ be defined by
\begin{align*}
\tilde{p}_{j,r}(\cb) = \frac{2q}{n} \sum_{s=1}^{n/2q} {\bar{\tilde{Y}}}_{2s-r}(W_j(\cb)),
\end{align*}
for $r \in \{0,1\}$, where ${\bar{\tilde{Y}}}_{2s-r} = \frac{1}{q} \sum_{t=(2s -r - 1)q+1}^{(2s-r)q} \tilde{X}_t$ is a measure in $\mathcal{M}(R,M)$, with total number of support points bounded by $q N_{\max}$, and remark that
\begin{align*}
\tilde{p}_j(\cb) = \frac{1}{2} \left ( \tilde{p}_{j,0}(\cb) + \tilde{p}_{j,1}(\cb) \right ). 
\end{align*}
Since $\E(\tilde{\bar{Y}}_{2s-r})(W_j(\cb)) = p_j(\cb)$, and the $\tilde{p}_{j,r}(\cb)$'s are sums of $ n/2q$ independent measures evaluated on $W_j(\cb)$, we may readily apply \cite[Lemma 22]{Chazal21} replacing $n$ by $n/(2q)$ to each of them, leading to the deviation bounds on the $\tilde{p}_j(\cb)$'s.

For the third inequality of Lemma \ref{lem:concentration_batch}, denoting by
\begin{align*}
\bar{\tilde{X}}_{n,j} = \frac{2q}{n} \sum_{s=1}^{n/2q} {\bar{\tilde{Y}}}_{2s-j}, 
\end{align*}
for $j \in \{0,1\}$, it holds, for any $\cb \in \B(0,R)^k$,
\begin{align*}
& \left \| \left ( \int (c_j - u)\1_{W_j(\cb)}(u) (\bar{\tilde{X}}_n - \E(X))(du)  \right )_{j=1, \hdots, k} \right \|  \\
& \qquad \leqslant \frac{1}{2} \left ( \left \| \left ( \int (c_j - u)\1_{W_j(\cb)}(u) (\bar{\tilde{X}}_{n,0} - \E(X))(du)  \right )_{j=1, \hdots, k} \right \| \right . \\
& \qquad \qquad  \left .  + \left \| \left ( \int (c_j - u)\1_{W_j(\cb)}(u) (\bar{\tilde{X}}_{n,1} - \E(X))(du)  \right )_{j=1, \hdots, k} \right \|\right ).
\end{align*}
Since each of the $\bar{\tilde{X}}_{n,j}$'s are i.i.d. sums of discrete measures (the ${\bar{\tilde{Y}}}_{2s-j}$'s), \cite[Lemma 22]{Chazal21} readily applies (with sample size $n/(2q)$), giving the result.
\end{proof}
We now proceed with the first term in \eqref{eq:onestep_connection_unmixing} as in \cite[Proof of Lemma 17]{Chazal21}. Using the first two inequalities of Lemma \ref{lem:concentration_batch} with $x=c_1 n p_{\min}/M$ yields a probability event $\Omega_1$ onto which
\begin{align*}
\tilde{p}_j(\cb) & \geqslant \frac{63}{64} p_j(\cb) - \frac{p_{\min}}{64} \geqslant \frac{31}{32} p_{\min}, \\
\tilde{p}_j(\cb) & \leqslant \frac{33}{32} p_j(\cb^*).
\end{align*}
Combining this with the inequality of Lemma \ref{lem:concentration_batch} yields, for $n$ large enough and all $\cb \in \B(\cb^*,R_0)$, with probability larger than $1-16e^{-c_1 n p_{\min}/n}-2e^{-x}$, 
\begin{align}\label{eq:contration_batch_iid}
\|\tilde{m}(\cb) - \cb^*\|^2 \leqslant 0.65 \|\cb - \cb^*\|^2 + \frac{C}{p_{\min}^2} D_{n/q}^2.
\end{align}
The precise derivation of \eqref{eq:contration_batch_iid} may be found in \cite[Proof of Lemma 17, pp.34-35]{Chazal21}. Plugging \eqref{eq:contration_batch_iid} into \eqref{eq:onestep_connection_unmixing} leads to, for a small enough $K$,
\begin{align*}
\|\hat{m}(\cb) - \cb^*\|^2 \leqslant \frac{3}{4} \|\cb - \cb^* \|^2 + \frac{C}{p_{\min}^2} D_{n/q}^2 + C_2 \|\hat{m}(\cb) - \tilde{m}(\cb)\|^2,
\end{align*}
with probability larger than $1-16e^{-c_1 n p_{\min}/n}-2e^{-x}$.

It remains to control the last term $\|\hat{m}(\cb) - \tilde{m}(\cb)\|^2$. To do so, note that, for every $j = 1, \hdots, j$, 
\begin{align}\label{eq:unmixing_weights}
\left | \hat{p}_j(\cb) - \tilde{p}_j(\cb) \right | & = \frac{1}{n} \left | \sum_{i=1}^n  X_i(W_j(\cb)) - \tilde{X}_i(W_j(\cb)) \right | \notag \\
& \leqslant \frac{M}{n} \sum_{i=1}^n \1_{X_i \neq \tilde{X}_i},
\end{align}
and 
\begin{align*}
\left \| (\bar{X}_n(du) - \bar{\tilde{X}}_n(du)) \left [ u \1_{W_j(\cb)}(u) \right ] \right \| & \leqslant \frac{2RM}{n} \sum_{i=1}^n \1_{X_i \neq \tilde{X}_i}.
\end{align*}
On $\Omega_1$, it holds, for every $j = 1, \hdots, k$,
\begin{align*}
\|\hat{m}_j(\cb) - \tilde{m}_j(\cb)\| & = \left \| \frac{\bar{X}_n(du) \left [u\1_{W_j(\cb)}(u) \right ]}{\hat{p}_j(\cb)} - \frac{\bar{\tilde{X}}_n(du) \left [u\1_{W_j(\cb)}(u) \right ]}{\tilde{p}_j(\cb)} \right \| \\
& \leqslant \left \| \bar{X}_n(du) \left [ u \1_{W_j(\cb)}(u) \right ] \right \| \left | \frac{1}{\hat{p}_j(\cb)} - \frac{1}{\tilde{p}_j(\cb)} \right | \\
& \qquad + \frac{1}{\tilde{p}_j(\cb)} \left \| (\bar{X}_n(du) - \bar{\tilde{X}}_n(du)) \left [ u \1_{W_j(\cb)}(u) \right ] \right \| \\
&  \leqslant R \frac{|\hat{p}_j(\cb) - \tilde{p}_j(\cb)|}{\tilde{p}_j(\cb)} + \frac{2RM}{n \tilde{p}_j(\cb)} \sum_{i=1}^n \1_{X_i \neq \tilde{X}_i} \\
& \leqslant C \frac{RM}{np_{\min}} \sum_{i=1}^n \1_{X_i \neq \tilde{X}_i}.
\end{align*}
Squaring and taking the sum with respect to $j$ gives the result.
\end{proof}

\subsubsection{Proof of Proposition \ref{prop:test_beta_mixing}}
\begin{proof}[Proof of Proposition \ref{prop:test_beta_mixing}]
Let $Z_1, \hdots, Z_n$ denote the sequence $\|\tilde{v}_1\|, \hdots, \|\tilde{v}_n\|$, that is a stationary $\beta$-mixing sequence of real-valued random variables. For $s \in \R$, we let 
\begin{align*}
F_n(t) = \frac{1}{n} \sum_{i=1}^n \1_{Z_i >t},
\end{align*}
and $F(t) = \P(Z >t)$, and $\hat{t}$ be such that $F_n(\hat{t}) \leqslant \alpha- \delta$. In the i.i.d. case, we might bound $\sup_t (F(t) - F_n(t))/\sqrt{F_n(t)}$ using a standard inequality such as in \cite[Section 5.1.2]{Boucheron05}.  

As for the proofs of Theorem \ref{thm:cv_batch} and \ref{thm:cv_minibatch}, we compare with the i.i.d. case by introducing auxiliary variables.

We let $\tilde{Z}_1, \hdots, \tilde{Z}_n$ be such that, denoting by $Y_k$ (resp. $\tilde{Y}_k$) the vector $(Z_{(k-1)q+1}, \hdots, Z_{kq})$ (resp.   $(\tilde{Z}_{(k-1)q+1}, \hdots, \tilde{Z}_{kq})$), it holds
\begin{itemize}
\item For every $1 \leqslant k \leqslant n/q$ $Y_k \sim \tilde{Y}_k$, and $\P(Y_k \neq \tilde{Y}_k) \leqslant \beta(q)$.
\item $(Y_{2k})_{k \geqslant 1}$ are independent, as well as $(Y_{2k-1})_{ k \geqslant 1}$. 
\end{itemize}
Let $\tilde{F}_n(t)$ denote $\sum_{i=1}^n \1_{\tilde{Z}_i >t}$. Then, for any $t \in \R$, we have
\begin{align*}
\tilde{F}_n(t) \leqslant F_n(t) + \frac{1}{n} \sum_{i=1}^n \1_{Z_i \neq \tilde{Z}_i}.
\end{align*}
If $\Omega_1$ is the event $\left \{ \frac{1}{n} \sum_{i=1}^n \1_{Z_i \neq \tilde{Z}_i} \leqslant  \sqrt{\frac{\alpha q}{n}} \right \}$, that has probability larger than $1- \beta(q) \sqrt{n/(\alpha q)}$ (using Markov inequality as before), then on $\Omega_1$ it readily holds
\begin{align*}
\tilde{F}_n(\hat{t}) & \leqslant F_n(\hat{t}) + \sqrt{\frac{\alpha q}{n}} \\
& \leqslant \alpha - \delta + \sqrt{\frac{\alpha q}{n}},
\end{align*} 
so that we may write, on the same event,
\begin{align}\label{eq:test_betamix_1}
F(\hat{t}) & = F(\hat{t})-\tilde{F}_n(\hat{t}) + \tilde{F}_n(\hat{t}) \notag \\
& \leqslant \alpha - \delta + \sqrt{\frac{\alpha q}{n}} + (F(\hat{t})-\tilde{F}_n(\hat{t})).
\end{align}
It remains to control the stochastic term $(F(\hat{t})-\tilde{F}_n(\hat{t}))$. To do so, we denote by
\begin{align*}
\tilde{X}_{j,0}(t) & = \frac{1}{q}\sum_{i=2(j-1)q +1}^{(2j-1)q} \1_{\tilde{Z}_i > t}, \\
\tilde{X}_{j,1}(t) & = \frac{1}{q}\sum_{i=(2j-1)q +1}^{2jq} \1_{\tilde{Z}_i > t}, 
\end{align*}
for $j \in [\![1,n/(2q)]\!]$ and $t \in \R$. Note that, for any $\sigma \in \{0,1\}$, $\tilde{X}_{j,\sigma}$'s are i.i.d., take values in $[0,1]$, and have expectation $F(t)$. Next, we define, for $1 \leqslant j \leqslant n/2q$ and $t \in \R$,
\begin{align*}
\tilde{F}_{n,0}(t) & = \frac{2q}{n} \sum_{j=1}^{n/2q} \tilde{X}_{j,0}(t) \\
\tilde{F}_{n,1}(t) & = \frac{2q}{n} \sum_{j=1}^{n/2q} \tilde{X}_{j,1}(t),
\end{align*}
and we note that $\tilde{F}_n(t) = \frac{1}{2}(\tilde{F}_{n,0}(t) + \tilde{F}_{n,1}(t))$. Since the $\tilde{F}_{n,\sigma}$'s are sums of i.i.d. random variables, the following concentration bound follows.
\begin{lem}\label{lem:concentration_detection}
For $j \in \{0,1\}$, and $x$ such that $(n/2q)x^2 \geqslant 1$, it holds 
\begin{align*}
\P \left ( \sup_{t \in \R} \frac{(F(t) - \tilde{F}_{n,j}(t))}{\sqrt{F(t)}} \geqslant 2x \right ) \leqslant 2ne^{-(n/2q)x^2}.
\end{align*}
\end{lem}
A proof of Lemma \ref{lem:concentration_detection} is postponed to the following Section \ref{sec:proof_lem_concentration_detection}. Now, choosing $x=2\sqrt{\frac{q \log(n)}{n}}$ entails that, with probability larger than $1-4(q/n) - \beta(q) \sqrt{n/(\alpha q)}$, 
\begin{align*}
F(\hat{t}) \leqslant \alpha - \delta + \sqrt{\frac{\alpha q}{n}} + 4\sqrt{\frac{q \log(n)}{n}}\sqrt{F(\hat{t})},
\end{align*} 
which leads to
\begin{align*}
\sqrt{F(\hat{t})} \leqslant 2\sqrt{\frac{q \log(n)}{n}} + \sqrt{\alpha - \delta + \sqrt{\frac{\alpha q}{n}} + 4 \frac{q \log(n)}{n}}.
\end{align*}
Choosing $\delta \geqslant 4 \sqrt{\alpha} \sqrt{\frac{q \log(n)}{n}} + \sqrt{\frac{\alpha q}{n}}$ ensures that the right-hand side is smaller than $\sqrt{\alpha}$.
\end{proof}

\subsubsection{Proof of Lemma \ref{lem:concentration_detection}}\label{sec:proof_lem_concentration_detection}
\begin{proof}[Proof of Lemma \ref{lem:concentration_detection}]
The proof follows the one of \cite[Lemma 22]{Chazal21} verbatim, at the exception of the capacity bound, that we discuss now. To lighten notation we assume that we have a $n/(2q)$ sample of $Y_i$'s, with $Y_i=(Z_{(i-1)q + 1}, \hdots, Z_{iq}) \in \R^q$, and we consider the set of functionals
\begin{align*}
\mathcal{F} = \left \{ y=(z_1, \hdots, z_q) \mapsto \frac{1}{q}\sum_{i=1}^q \1_{z_i > t} \right  \}.
\end{align*}  
Following \cite[Lemma 22]{Chazal21}, if $S_\mathcal{F}(y_1, \hdots, y_{n/(2q)})$ denotes the cardinality of $\left \{ f(y_1), \hdots, f(y_{n/(2q)}) \mid f \in \mathcal{F} \right \}$, we have to bound
\begin{align*}
S_\mathcal{F}(Y_1, \hdots, Y_{n/(2q)},Y'_1, \hdots, Y'_{n/(2q)} ),
\end{align*} 
where the $Y'_{i}$'s are i.i.d. copies of the $Y_i$'s. Since, for every $y_1, \hdots, y_{n/q}$, recalling that $y_i=(z_{(i-1)q + 1}, \hdots, z_{iq})$, it holds
\begin{align*}
S_\mathcal{F}(y_1, \hdots, y_{n/q}) \leqslant \left \{ (\1_{z_1 > t}, \hdots, \1_{z_n >t}) \mid t \in \R \right \},
\end{align*} 
we deduce that
\begin{align*}
S_\mathcal{F}(Y_1, \hdots, Y_{n/(2q)},Y'_1, \hdots, Y'_{n/(2q)} ) \leqslant n.
\end{align*}
The remaining of the proof follows verbatim \cite[Lemma 22]{Chazal21}.
\end{proof}

\bibliography{biblio}

\end{document}